\documentclass[12pt]{amsart}
\usepackage{preamble}

\begin{document}

\title[Divergence-Kernel method for linear responses and generative models]{
Divergence-kernel method for linear responses of densities and generative models
}

\begin{abstract}
We derive the divergence-kernel formula for the linear response of random dynamical systems. 
Specifically, the pathwise expression is for the parameter-derivative of the marginal or stationary density, not an averaged observable.
Our formula works for multiplicative and parameterized noise over any period of time; it does not require hyperbolicity. 
Then we derive a Monte-Carlo algorithm for linear responses.

We develop a new framework of generative models, DK-SDE, where the model is a parameterized SDE, that 
(1) directly uses the KL divergence between the empirical data distribution and the marginal density of the SDE as the training objective, and
(2) accommodates parametrizations in both drift and diffusion over a long time span, allowing prior structural knowledge to be incorporated explicitly.
The optimization is done by gradient-descent enabled by the divergence-kernel method, which involves only forward processes and therefore substantially reduces memory cost.
We demonstrate the new model on a 20-dimensional Lorenz system.



\end{abstract}

\maketitle

\section{Introduction}
\label{s:intro}

\subsection{Main results}

This paper rigorously derives the divergence-kernel formula for the linear response in discrete-time random dynamical systems.
It is a pathwise expression of the derivative of marginal or stationary density with respect to parameters in the drift or diffusion coefficients.
Then we \textit{formally} pass to the continuous-time limit in \Cref{t:divker_sde}.

\begin{formula}[divergence-kernel formula for linear responses of SDEs]
\label{t:divker_sde}
Consider the SDE in $\R^M$ parameterized by $\gamma$,
\[
dx^\gamma_t = F^\gamma_t(x^\gamma_t) dt + \sigma^\gamma_t(x^\gamma_t)dB_t,
\]
For any backward adapted process $\{ \alpha_t \}_{ t=0 }^T$, let $\{ \nu_t \}_{ t=0 }^T$ be the forward covector process at $\gamma=0$ such that
\begin{equation*}\begin{split}
  d \nu
  = \left( 
  (\nabla \sigma \nabla \sigma^T 
  - \nabla F^T - \alpha)\nu 
  - \nabla \div F 
  + \nabla^2 \sigma \nabla \sigma 
  + \nabla \sigma \Delta \sigma \right) dt
  - 
  (\nabla \sigma \nu^T 
  + \nabla^2 \sigma 
  + \alpha \sigma^{-1}) dB
\end{split}\end{equation*}
with initial condition $ \nu_0 = \nabla \log h_0(x_0)$,
where $h_t$ is the probability density of $x_t$.
Then
$ \nabla \log h_t (x_t) = \E{\nu_t  \middle| x_t}$.
Moreover, let $\delta (\cdot) :=\partial (\cdot)/ \partial \gamma$, then
\begin{equation*} \begin{split}
  \delta \log h^\gamma_T (x_T) 
  = 
  \mathbb{E}\left[
  \delta \log h^\gamma_0 
  + \int_{t=0}^{T}
  \delta \sigma^\gamma \nabla \sigma^T \nu 
  - \delta F^{\gamma T} \nu
  - \div \delta F^\gamma 
  + \delta \sigma^\gamma \Delta \sigma 
  + \nabla \sigma^T \nabla \delta \sigma^\gamma dt
  \right.
  \\
  \left.
  - \int_{t=0}^{T}
  (\delta \sigma^\gamma  \nu^T
  + \nabla \delta \sigma^{\gamma T}) dB
  \middle| x_T\right].
\end{split} \end{equation*}
\end{formula}

Here $\Delta \sigma$ is the Laplacian of $\sigma$, $B$ is the Brownian motion, and the SDE is Ito.
In this paper $\sigma$ is assumed to be a scalar function uniformly bounded away from 0, and extension to matrices is conceptually straightforward.
The time integrations are the limits of \Cref{e:backito,e:siye}.
For large $T$, we can just pick $\alpha$ larger than the magnitude of the strongest contraction rate, which stabilizes $\nu$ from growing exponentially fast.
In fact, for memory-light implementations we restrict to $\alpha_t=\alpha(x_t)$ or constant $\alpha$; the backward-adapted freedom is mainly a theoretical gauge.
The conditional expectation is taken over all paths ending at the same point; it turns pathwise quantities into terminal density derivative.
The main advantages of this formula are:
\begin{enumerate}
  \item Compared with the path-kernel method in \cite{dud,apk}, it computes derivative of the marginal density, not an averaged observable.
  Moreover, it involves forward processes so the memory cost is lower than the adjoint path-kernel method, yet it still computes the full gradient with cost independent of the number of parameters.
  \item Compared with the finite-element method in \cite{FP25optimalRes}, its main expression is pathwise, so we can compute it via Monte Carlo algorithms in high dimensions.
  \item Compared with the likelihood ratio method in \cite{Glynn1990}, the diffusion coefficient $\sigma$ can depend on the parameter $\gamma$ and the location $x$.
  \item Compared with the fast response method in \cite{fr}, it does not require hyperbolicity (in the sense of dynamical systems). 
  \item The size of the expression does not grow exponentially to $T$, so we can easily extend our result to stationary distributions (see \Cref{t:ergodic}).
  \item Compared with the author's divergence-kernel method for scores in \cite{divKer}, this paper considers the parameter-gradients, which enables a new framework of parametric SDE generative models.
\end{enumerate}
This formula gives an efficient gradient tool for diffusion processes with not too small noise or not too strong contraction.
In \Cref{s:eg_lorenz}, the corresponding Monte-Carlo algorithm computes the linear response of the marginal density of a Lorenz 96 system, which is difficult for other methods.

With the new divergence-kernel tool, we develop a new framework for generative models.
The logic is rather simple. 
Our model is an SDE, whose drift and diffusion depends on the state and (many) undetermined parameters.
The loss function is the KL-divergence of the empirical distribution of the data with respect to the marginal distribution of the SDE.
With the new gradient tool, we can just optimize parameters in the SDE model via stochastic gradient descent.

This new framework combines several advantages of different current generative models, and also offers some new advantages:
\begin{enumerate}
  \item Directly uses the KL divergence between the empirical data distribution and the marginal density of the SDE as the training objective.
  \item Accommodates parametrizations in both drift and diffusion, so prior knowledge of the SDE model can be easily incorporated into the model and reduce the complexity (number of parameters) of the model.
  \item Allows deep models (large $T$) without imposing normalizing requirements.
  \item It only saves the current state. In particular, we do not have a backward procedure, so we do not save the entire path, substantially reducing the memory cost. 
  \item If we use the stationary measure of the SDE model in the framework, then the ergodic formula in \Cref{t:ergodic} further improves efficiency.
\end{enumerate}
Finally, we might achieve better overall efficiency in general situations; but this requires many more experiments to verify.

\subsection{Literature review}

\subsubsection{Linear response methods}
\label{s:linear_response}

The marginal or stationary distribution of a random dynamical system is of central interest in applied sciences.
We are interested in the parameter-derivative $\delta \log h_t$, where $h_t$ is the probability density at time $t$.
There are three basic methods for expressing and computing the linear response of random dynamical systems: the path-perturbation method (shortened as the path method), the divergence method, and the kernel-differentiation method (shortened as the kernel method).
The relation and difference among the three basic methods can be illustrated in a one-step system, which is explained in \cite{Ni_kd}.

The path-perturbation method computes the path-perturbation under fixed randomness (such as Brownian motion), then averages over many realizations of randomness.
This is also known as the ensemble method or the stochastic gradient method \cite{eyink2004ruelle,lucarini_linear_response_climate}.
It also includes the backpropagation method, which is the basic algorithm for machine learning.
However, it is expensive for chaotic or unstable systems; the work-around is to artificially reduce the size of the path-perturbation, such as shadowing or clipping methods \cite{Ni_nilsas,Ni_NILSS_JCP,clip_gradients2,FP25}, but they all introduce systematic errors.
Moreover, it computes the derivative of an averaged observable, but does not give the derivative of the probability density.

The divergence method is also known as the transfer operator method, since the perturbation of the measure transfer operator is some divergence \cite{Galatolo2014,Wormell2019a,Zhang2020}.
However, it is expensive for high-dimensional systems with contracting directions.
For SDEs, we can take the derivative of the Fokker-Planck equation, which involves the divergence of the drift and a second-order differential operator \cite{SL22}, which is hard to sample on paths.

The kernel-differentiation method works only for random systems; as the name indicates, the derivative now hits the probability kernel.
It is also known as the Cameron-Martin-Girsanov theorem \cite{CM44, MalliavinBook}, the likelihood ratio method \cite{Rubinstein1989,Reiman1989,Glynn1990}, or the Monte-Carlo gradient method \cite{Caflisch2021}.
However, it cannot handle multiplicative noise or perturbation on the diffusion coefficients.
It is also expensive when the noise is small.
We gave an ergodic version and a foliated version in \cite{Ni_kd}.

Mixing two basic methods can overcome some of their major shortcomings.
Our path-divergence method (also known as the fast response formula) gives the pointwise expression of linear responses of hyperbolic deterministic chaos \cite{Ni_asl,TrsfOprt,fr,Ni_nilsas,far,vdivF}.
It is good at high-dimensional and low-noise to deterministic systems.
It can be used to solve the optimal response problem in the hyperbolic setting \cite{GN25}. 
However, it does not work when the hyperbolicity is poor \cite{Baladi2007,wormell22}, and only the unstable part involves derivatives of densities.
The path-divergence idea was also used in the proof of the basic path-perturbation or divergence linear response formulas for hyperbolic systems \cite{Dolgopyat2004,Ruelle_diff_maps,Jiang2012,Gouezel2008,Baladi2017}.

Our path-kernel method in \cite{dud} gives the pathwise linear response formula of the diffusion coefficients.
It works well when the noise level is not too small or the largest Lyapunov exponent is not too large; it also does not require hyperbolicity.
The adjoint version solves a difficult version of variational data assimilation problem in chaotic systems with unknown parameters and partial observation, where the loss is a single long-time functional accounting for discrepancies in both the observations and the dynamics.
However, it can be expensive when the noise is small and the largest Lyapunov exponent is large.
Moreover, only part of the expression is derivative on the terminal densities.
The path-kernel idea was also used in the Bismut-Elworthy-Li formula for the derivative with respect to the initial conditions \cite{Bismut84,EL94,PW19bismut}.

Our divergence-kernel method in \cite{divKer} gives the score (spatial-derivative of marginal density) of random dynamical systems and SDEs.
Compared to previous pathwise expressions for scores \cite{score14,Liu2025,MG25score}, the main advantage is that it works for multiplicative noise efficiently.
It is good at systems with not too much contraction or not too little noise.
However, it can be expensive when the noise is small and the contraction is strong.
It fits into the conventional framework of score-based diffusion models, and it enables us to use nonlinear SDEs with multiplicative noise as the forward model, since it can compute the scores efficiently.
This paper extends this result to linear responses, which further enables a new framework of generative models.

\subsubsection{Generative models}

Our new linear response formula opens up the possibilities of redesigning the current framework of diffusion models.
At a high level, the bottleneck in `parametric SDE as a generative model' is that the most direct training loss for i.i.d.\ data is the KL-divergence.
But optimizing this loss requires differentiation of the model's marginal density.
For general SDEs this was typically intractable, where many classical tricks either break or become prohibitively high-variance.
So previously we had to either replace the training loss or simplify the SDE.

A first workaround is to keep using a parameterized SDE to learn a distribution, but replace KL-divergence training loss by the discrepancy over expectations of certain observables; this covers adversarial critics, kernel MMD, and moment matching \cite{Arjovsky2017,Gretton2012,Li2015}.
This new loss does not require evaluating the terminal marginal density or its parameter-gradient, so we can use the standard backpropagation to compute the gradient of this loss.
However, in chaotic or unstable settings, we need to replace conventional backpropagation method by new methods such as the adjoint path-kernel method \cite{apk}.
More importantly, the fit of the data is limited by the chosen observable or the class holding the observable: the user needs to introduce an additional modeling choice (the observable/critic/kernel/features/weights) that mediates what aspects of the distribution are matched.

A second workaround, prominent in diffusion models \cite{diffusion_model} and flow-matching methods \cite{Lipman2023,Zhou2025}, is to introduce an auxiliary probability path: a designer-chosen time-indexed family of intermediate probability distributions $\{H_t\}_{t\in [0,T]}$ (typically $T=1$) connecting a simple base distribution $H_0$ to the data distribution $H_T$.
The main object to be learned now is the score or a velocity field, and the corresponding training loss becomes the score/denoising loss or vector-field regression loss.
Similar to the first workaround, this training loss can be differentiated by standard backpropagation.
However, the auxiliary probability path is typically decided without consulting the prior structural knowledge about the model, so the prior knowledge does not automatically translate into computational savings.
Even if we have a very simple SDE with a few parameters, the score is not necessarily simple, so the model for the score still needs a large network and potentially many training steps.

A third workaround is to keep using the KL divergence as the loss, but restrict the model class so that the dependence of the distribution on the parameters becomes tractable. 
We give three examples. 
First, autoregressive models factorize the joint density as a product of conditional densities under a chosen ordering; this introduces an additional modeling choice (the ordering) and typically requires causal/masked architectural constraints to parameterize these high-dimensional conditionals \cite{VanDenOord2016}. 
Second, in normalizing flows \cite{Dinh2017} one represents the model distribution through an invertible transport map, so that the change-of-variables formula expresses the density via a Jacobian determinant.
Gradients of the resulting log-determinant terms are closely related to divergence-type expressions, connecting to the divergence (transfer-operator) method in linear response (see the review in \Cref{s:linear_response}); in practice, moderately deep flows often require additional constraints or regularization to maintain good conditioning.
Third, one may fix the diffusion coefficient, which makes likelihood-based gradients more tractable via likelihood-ratio/Girsanov-type estimators or related score-function identities. However, this narrows the model class: if the diffusion cannot be optimized, the induced SDE marginals may remain far from the data distribution.

Summarizing, our divergence-kernel method enables the gradient tool to optimize KL divergence with respect to the diffusion coefficient or under multiplicative noise.
With this, we can work with the simplest generative model, without introducing an auxiliary probability path or an auxiliary observable and without freezing the diffusion.
This setting is common rather than exceptional: in mechanistic SDE modeling one often wants likelihood-style marginal fitting while allowing parameters in the diffusion; existing approaches typically avoid this combination.

\subsection{Structure of paper}

\Cref{s:notations} defines some basic notation and states some basic assumptions on regularities.
\Cref{s:div1step} derives our the divergence formula for the linear response of one-step random systems.
\Cref{s:divdiscrete} derives our results for discrete-time random dynamical systems by combining the one-step divergence linear response formula with the multi-step divergence-kernel formula for scores.
\Cref{s:divcts_prep} recalls some useful facts from \cite{divKer} about scores for SDEs.
\Cref{s:divkerCts} formally passes the linear response formula to the limit of SDEs.
\Cref{s:howtouse} discusses about how to use this result.

\Cref{s:numeric} uses the new formulas to numerically compute the linear responses of two model systems.
\Cref{s:eg_1dim} considers a 1-dimensional process with multiplicative noise.
For this 1-d example, we can plot the parameter-derivative of the density.
\Cref{s:eg_lorenz} computes the linear response of the stationary measure of the 40-dimensional Lorenz 96 model with multiplicative noise.
These examples cannot be solved by previous methods since the system is high-dimensional, has multiplicative noise, has contracting directions, and is non-hyperbolic.
Moreover, compared with our path-kernel methods, here we can compute the derivative of the density.

\Cref{s:diffmodel} proposes a parametric SDE generative model, DK-SDE (DK is short for divergence-kernel), which involves only two forward processes.
It uses the divergence-kernel formula (\textit{not} the backpropagation method, which belongs to the path-perturbation method) to perform a gradient descent on the parameters of an SDE model.
Its loss function is the KL-divergence between the data and the marginal distribution of the SDE.
The logic is more straightforward than the current diffusion models, and the memory cost is less.

\Cref{s:kd} briefly goes through the kernel-differentiation method and explains its problem when applied to linear responses of high-dimensional systems with respect to scalar diffusion coefficients.
We also derive the conventional likelihood ratio formula for additive noise systems using our notation.

\section{Notations and assumptions}
\label{s:notations}

We define some geometric notations for spatial derivatives.
In this paper, the ambient space is the $M$-dimensional Euclidean space $\R^M$.
We denote both vectors and covectors by column vectors in $\R^M$.
The product between a covector $\nu$ and a vector $v$ is denoted by $\cdot$, 
that is,
\begin{equation*}\begin{split}
  \nu\cdot v 
  := v\cdot \nu
  :=\nu^T v
  :=v^T \nu.
\end{split}\end{equation*}
Here $v^T$ is the transpose of matrices or (co)vectors.
Then we use $\ip{,}$ to denote the inner-product between two vectors
\begin{equation*}\begin{split}
  \ip{u,v}:=u^T v.
\end{split}\end{equation*}
Note that $\DB$ may be either a vector or a covector.

The score function is $\nabla \log{h_T} = \frac{\nabla h_T}{h_T}$, where 
\begin{eqnarray*}
  \nabla (\cdot):=\pp{(\cdot)}x,\quad
  \nabla_v (\cdot):=
  \nabla (\cdot) v :=
  \pp{(\cdot)}x v,
\end{eqnarray*}
Here $\nabla_YX$ denotes the (Riemann) derivative of the tensor field $X$ along the direction of $Y$.
It is convenient to think that $\nabla$ always adds a covariant component to the tensor.
The divergence of a vector field $X$ is defined by
\[ \begin{split}
  \div X:= \sum_{i=1}^M \nabla_{e_i} X,
\end{split} \]
where $e_i$ is the $i$-th unit vector of the canonical basis of $\R^M$.

For a map $g$, let $g_*$ be the Jacobian matrix, or the pushforward operator on vectors; readers may also denote this by $\nabla g$.
We define the covector $\div g_*$, the divergence of the Jacobian matrix $g_*$, as
\begin{equation} \begin{split} \label{e:zhao}
  \div g_* :=
  \nabla \log \left|  g_{*} \right|:=
  \frac{\nabla \left|  g_{*} \right|}{\left| g_{*} \right|},
\end{split} \end{equation}
where $\left| g_{*} \right|$ is the determinant of the matrix.
It is called a divergence due to an equivalent definition by contracting the Hessian (see \cite{TrsfOprt}).
Note that this is not the row-wise (or column-wise) divergence of a matrix.

For discrete-time results in \Cref{s:divdiscrete}, we assume that 
\begin{assumption*} 
  The density of the initial distribution $h_0$, the probability kernel $k$, drift $f$, and diffusion $\sigma$ are $C^1$ jointly in $\gamma$ and $x$.
  The diffusion field $\sigma^\gamma(x)$ is uniformly bounded away from zero.
\end{assumption*}
Our derivation is rigorous for the discrete-time case, but is \textit{formal} when passing to the continuous-time limit, so we do not list the technical assumptions here, and leave the rigorous proof to a later paper.
We just assume that all integrations, averages, and change of limits are legit.
All numerical experiments use the time-discretized formulas; the continuous-time statements serve as guiding limits.

\section{Derivation of the divergence-kernel formula for linear responses} \label{s:div}

There are two basic ingredients to our work.
The first is the divergence formula for $\delta \log h^\gamma_1$ in the one-step system, which we shall derive now.
The second is the divergence-kernel formula for the scores in many-step systems and its continuous-time limit; this result is from \cite{divKer}.

\subsection{Divergence formula One-step system} \label{s:div1step}

Let $x^\gamma_0\sim h^\gamma_0$, $b^\gamma_0\sim k$, the one-step dynamics is 
\begin{eqnarray*}
  x^\gamma_1 = f^\gamma_0(x^\gamma_0) + \sigma^\gamma_0 (x^\gamma_0) b_0.
\end{eqnarray*}
Here $\gamma$ is the parameter; when we omit $\gamma$ in our notation, it takes the default value $0$.
We sometimes omit the subscripts that indicate the time step of $f$ and $\sigma$, which are the same as the time step of the variable $x$.
Let $h^\gamma_1$ denote the density of $x^\gamma_1$, which has a natural expression:
\begin{equation}\label{e:h1}
  h^\gamma_1(x_1) = \int h^\gamma_0(x_0)p^\gamma_0(x_0,x_1) dx_0.
\end{equation}
where the probability kernel is
\begin{equation*}\begin{split}
  p^\gamma_0(x_0,x_1) := (\sigma^\gamma_0) ^{-M}(x_0) 
  k\left(\frac{x_1-f^\gamma_0(x_0)}{\sigma^\gamma_0(x_0)}\right).
\end{split}\end{equation*}

Change the variable $x_0$ to $b_0=(x_1-f^\gamma_0(x_0)) / \sigma^\gamma_0(x_0)$ under fixed $x_1$.
Define
\begin{equation*}\begin{split}
  x_1 = g^\gamma_{b_0}(x_0) := f^\gamma_0(x_0) + \sigma^\gamma_0(x_0) b_0
\end{split}\end{equation*}
Since $x_1$ is fixed, $x_0$ and $b_0$ are related by
\begin{equation*}\begin{split}
  0 = dx_1 = g^\gamma_{b_0*}(x_0) dx_0
  + \sigma^\gamma_0(x_0) d b_0,
\end{split}\end{equation*}
where $g^\gamma_{b_0*} := \nabla f^\gamma + b_0 (\nabla \sigma^\gamma) ^T$ is the pushforward operator of $g_{b_0}$.
Changing the variable causes a nontrivial Jacobian determinant.
\[ \begin{split}
  \left| \dd {b_0}{x_0} \right|
  = (\sigma_0^\gamma)^{-M}(x_0) \left| g^\gamma_{b_0*} (x_0) \right|,
\end{split} \]
where $\left| g^\gamma_{b_0*} \right| $ is the determinant.
So we can express $h^\gamma_1$ by an integration over $b_0$,
\begin{equation} \begin{split} \label{e:h1_b}
  h^\gamma_1(x_1) = \int h^\gamma_0(x^\gamma_0)k(b_0) \left|  g^\gamma_{b_0*} (x^\gamma_0) \right| ^{-1} db_0,
  \quad \textnormal{where} \quad 
  x^\gamma_0 = (g^\gamma_{b_0})^{-1}(x_1).
\end{split} \end{equation}
For now, we assume $g^\gamma_{b_0}$ is injective, since this is the case when we discretize SDEs; if it is $n$-to-1, then we need to further sum over its preimage.

\begin{lemma} [Divergence formula for one-step linear response] \label{t:div1step}
\begin{equation*} \begin{split}
  \delta \log h^\gamma_1 (x_1) 
  = \E{
  \frac{\delta h^\gamma_0 + \nabla_{\delta x^\gamma_0} h_0}{h_0}(x_0)
  - \frac{\delta\left| g^\gamma_{b_0*} \right| 
  + \nabla_{\delta x^\gamma_0} \left|  g_{b_0*} \right|}
  {\left| g_{b_0*} \right|}(x_0)
  \middle| x_1},
\end{split} \end{equation*}
where $ \delta x^\gamma_0 =- g^{-1}_{b_0*} \delta g^\gamma_{b_0}(x_0).  $
\end{lemma}

\begin{proof}
Differentiate \Cref{e:h1_b} by $\gamma$, note that the dummy variable $b_0$ is not differentiated; the derivative hits $h^\gamma_0$, $ g^\gamma_{b_0*}$, and $x^\gamma_0$.
We get
\begin{equation*} \begin{split}
  \delta h^\gamma_1 (x_1) 
  = \int \left(\frac{\delta h^\gamma_0 + \nabla_{\delta x^\gamma_0} h_0}{h_0}(x_0)
  - \frac{\delta\left| g^\gamma_{b_0*} \right| 
  + \nabla_{\delta x^\gamma_0} \left|  g_{b_0*} \right|}
  {\left| g_{b_0*} \right|}(x_0)
  \right)
  h_0(x_0)k(b_0) \left| g_{b_0*} (x_0) \right| ^{-1} db_0,
\end{split} \end{equation*}
To get $\delta x^\gamma_0$, differentiate $x_1 = g^\gamma_{b_0}(x^\gamma_0)$, note that $x_1$ is fixed, we get
\begin{equation*}\begin{split}
  0 = \delta g^\gamma_{b_0}(x_0)
  + g_{b_0*}\delta x^\gamma_0.
\end{split}\end{equation*}
Hence, $\delta x^\gamma_0 =- g^{-1}_{b_0*} \delta g^\gamma_{b_0}(x_0)$.
Then substitute to prove the lemma.
\end{proof}

We do not work on the divergence-kernel method for linear responses in one-step, since it is not directly used.
The reason is that the kernel formula for linear responses (see \Cref{t:ker1step} in \Cref{s:kd}) has a term with factor $M$, which can be expensive (since large terms require many samples to compute averages) for important applications.
The kernel formula for scores does not have this large term.
Hence, we choose to use the divergence formula for the linear response and the divergence-kernel formula for the scores.

That being said, we can certainly mix the kernel and divergence linear response formulas in one-step, then repeatedly apply to many-step systems, and then pass to the continuous-time limit.
That approach might be beneficial if we later work with anisotropic noise, where $\sigma$ is a matrix, and each parameter controls only a low-dimensional structure, so we do not have the $M$ factor. 
Another scenario where this approach might be more useful is that the noise is low-dimensional.
Both situations incur other difficulties, so we will defer to later papers.

\subsection{Discrete-time divergence-kernel formula for linear responses} \label{s:divdiscrete}

Starting from the initial density $x^\gamma_0 \sim h^\gamma_0$, consider the discrete-time random dynamical system
\begin{equation} \begin{split} \label{e:discreteEq}
  x^\gamma_{n+1} = f^\gamma_n(x^\gamma_n) + \sigma^\gamma_n(x^\gamma_n)b_n,
\end{split} \end{equation}
where $b_n$ is i.i.d.\ distributed according to the kernel density $k$.
Denote the density of $x^\gamma_n$ by $h^\gamma_n$.
The system runs for a total of $N$ steps.

We use the divergence formula in one-step, then substitute the recent divergence-kernel formula for the score at $\gamma=0$.
Note that the score formula contains the kernel-differentiation formula for the score, but not the kernel-differentiation formula for linear responses.
Hence, we do not have the term involving $M$ which can potentially get very large for scalar diffusion coefficients.
First, recall the following:

\begin{definition}
\label{d:backward}
  We say that a process $\{ \alpha_n \}_{ n=0 }^N$ is \textit{backward adapted}, or adapted to the backward filtration, if $\alpha_n$ belongs to the $\sigma$-algebra $\sigma(x_n,\ldots, x_N)$.
\end{definition}

\begin{lemma} [N-step forward divergence-kernel formula for score \cite{divKer}] \label{t:divkernstep_score}
For any backward adapted process $\{ \alpha_n \}_{ n=0 }^N$,
let $\{ \nu_n \}_{ n=0 }^N$ be the forward covector process
\begin{equation*}\begin{split}
\nu_0 = \nabla \log h_0(x_0),
\qquad
\nu_{n+1} = (1-\alpha_{n+1}) g_{b_n}^{*-1}
\left( \nu_{n} - \div g_{b_{n}*} (x_{n}) \right)
+ \frac {\alpha_{n+1} \nabla \log k (b_{n})} {\sigma_n(x_{n})} .
\end{split}\end{equation*}
Then for any $N\ge0$,
\begin{equation*}\begin{split}
  \nabla \log h_N (x_N) 
  = \E{\nu_N  \middle| x_N}.
\end{split}\end{equation*}
\end{lemma}

\begin{theorem} [N-step divergence-kernel formula for linear responses] \label{t:divkernstep}
\begin{equation*} \begin{split}
  \delta \log h^\gamma_N (x_N) 
  = \E{
  \delta \log h^\gamma_0 
  + \sum_{n=0}^{N-1}
  \left( \nu_n - \nabla \log |g_{b_n*}| \right) \cdot \delta x^\gamma_n
  - \delta \log | g^\gamma_{b_n*} | 
  \middle| x_N}.
\end{split} \end{equation*}
where $ \delta x^\gamma_n =- g^{-1}_{b_n*} \delta g^\gamma_{b_n}(x_n).  $
\end{theorem}

\begin{proof}
By \Cref{t:div1step}, denote $ \delta x^\gamma_{N-1} :=- g^{-1}_{b_{N-1}*} \delta g^\gamma_{b_{N-1}}(x_{N-1})$,
\begin{equation*} \begin{split}
  \delta \log h^\gamma_N (x_N) 
  = \E{
  \delta \log h^\gamma_{N-1} 
  + \nabla_{\delta x^\gamma_{N-1}} \log h_{N-1}
  - \delta \log | g^\gamma_{b_{N-1}*} | 
  - \nabla_{\delta x^\gamma_{N-1}} \log |g_{b_{N-1}*}|
  \middle| x_N},
\end{split} \end{equation*}
where all quantities are evaluated at $x_{N-1}$.

To handle the term involving $\nabla \log h_{N-1}$, note that it does not depend on $x_N$, so
\begin{equation*}\begin{split}
  \E{\nabla_{\delta x^\gamma_{N-1}} \log h_{N-1} (x_{N-1}) \middle| x_N}
  = 
  \E{\E{\nabla \log h_{N-1} (x_{N-1}) \middle| x_{N-1},x_N}\cdot \delta x^\gamma_{N-1} \middle| x_N}.
\end{split}\end{equation*}
use \Cref{t:divkernstep_score} (note that it requires conditioning on all steps $\ge N-1$),
\begin{equation*}\begin{split}
  \E{\nabla \log h_{N-1} (x_{N-1}) \middle| x_{N-1},x_N}
  = 
  \E{\nu_{N-1} \middle|  x_{N-1},x_N}.
\end{split}\end{equation*}
Since $\delta x^\gamma_{N-1}$ is measurable with respect to $(x_{N-1}, x_N)$, 
\begin{equation*}\begin{split}
  \E{\nabla_{\delta x^\gamma_{N-1}} \log h_{N-1} (x_{N-1}) \middle| x_N}
  = \E{
  \E{\nu_{N-1}\cdot 
  \delta x^\gamma_{N-1} \middle| x_{N-1}, x_N}\middle| x_N}
  =
  \E{\nu_{N-1} \cdot \delta x^\gamma_{N-1} \middle| x_N}.
\end{split}\end{equation*}
By substitution,
\begin{equation*} \begin{split}
  \delta \log h^\gamma_N (x_N) 
  = \E{
  \delta \log h^\gamma_{N-1} 
  - \delta \log | g^\gamma_{b_{N-1}*} | 
  + \left( \nu_{N-1} - \nabla \log |g_{b_{N-1}*}| \right)
  \cdot \delta x^\gamma_{N-1}
  \middle| x_N}.
\end{split} \end{equation*}
Repeat this on $\delta \log h^\gamma_n (x_n)$ for each $n$ to prove the lemma.
\end{proof}

Note that we need the score of $h_0$, or $\nu_0$.
Only for large $T$ and systems with fast decorrelation, we can use any $\nu_0$ and the impact would be negligible.
See \Cref{t:ergodic} for more details.

\subsection{Preparations for SDEs} \label{s:divcts_prep}

This subsection introduces some notation for time-discretized SDEs and recalls some useful facts from \cite{divKer} about scores for SDEs.
We want to \textit{formally} pass the discrete-time results to the continuous-time limit SDE,
\begin{equation} \begin{split} \label{e:sde}
dx^\gamma_t = F^\gamma_t(x^\gamma_t) dt + \sigma^\gamma_t(x^\gamma_t)dB_t,
\end{split} \end{equation}
where $B$ denotes a standard Brownian motion.
All SDEs in this paper are in the Ito sense, so the discretized version takes the form of \Cref{e:discreteSDE}.
Denote the density of $x_t$ by $h_t$.
Let $\cF_t$ be the $\sigma$-algebra generated by $\{x_\tau\}_{\tau\le t}$.

Our derivation is performed on the time span divided into small segments of length $\Dt$.
Let $N$ be the total number of segments, so $N\Dt = T$.
Denote
\[ 
  \DB_n:= B_{n+1} - B_n.
\]
Denote $(\cdot)_n = (\cdot)_{n\Dt}$.
The time-discretized SDE is 
\begin{equation} \label{e:discreteSDE}
x^\gamma_{n+1} - x^\gamma_{n}
=F^\gamma(x^\gamma_n) \Dt + \sigma^\gamma(x^\gamma_n)\DB_n.
\end{equation}

This section derives score formulas for time-discretized SDEs, then formally pass to the continuous-time limit.
Comparing the notations, we have
\begin{equation}\begin{split}
\label{e:lifu}
  f^\gamma_n(x) := x+ F^\gamma_n(x)\Dt,
  \quad 
  b_n := \DB_n,
  \quad 
  \\
  k(b) := \left(2\pi \sqrt \Dt \right)^{-M/2} \exp{ -\frac {\ip{b,b}}{2\Dt}} ,
  \quad 
  \nabla \log k \left(\DB_n\right)
  = - \frac {\DB_n}{\Dt},
  \\
  g^\gamma_{b_n}(x_n) 
  := x_n + F^\gamma_n(x_n)\Dt + \sigma^\gamma_n(x_n) \DB_n,
  \\
  g^\gamma_{b_n *}(x_n) 
  = I + \nabla F^\gamma_n(x_n)\Dt +  \DB_n \nabla \sigma_n^{\gamma T} (x_n).
\end{split}\end{equation}
Here $\nabla F$ and $\DB \nabla \sigma^T$ are matrices, whose $i$-th row $j$-th column are
\begin{equation*}\begin{split}
  \left[\nabla F \right]_{ij}
  = \pp {F^i} {x^j}
  =: \partial_j F^i,
\quad \quad
  \left[\DB \nabla \sigma^T \right]_{ij}
  = \pp {\sigma} {x^j} \DB^i
  =: \partial_j \sigma \DB^i.
\end{split}\end{equation*}

We keep only terms of order $\Dt$ in expectation and drop higher-order terms since they integrate to zero; this is standard in Ito calculus.
Ignoring terms that go to zero in the continuous-time limit, \cite{divKer} derived explicit expressions of some terms for time-discretized SDEs,
\begin{equation}\begin{split}\label{e:haiyu}
  \left| g^\gamma_{b_n*} \right|
  \approx
  1 + \sum_{i=1}^M \left(\partial_i F_n^{\gamma i} \Dt + \partial_i \sigma_n^\gamma \DB^i\right)
  =
  1 +  \div F_n^\gamma \Dt 
  + \nabla \sigma_n^\gamma \cdot \DB,
\\
  \div g_{b_n *} (x_n)
  :=
  \nabla \log \left|  g_{b*} \right|
  \approx 
  \nabla \div F_n(x_n) \Dt 
  + \nabla^2 \sigma_n(x_n) \DB
  - \nabla^2 \sigma_n \nabla \sigma_n (x_n) \Dt,
\end{split}\end{equation}
where $ \nabla^2 \sigma \DB = \sum_i \nabla \partial_i \sigma(x_n) \DB_n^i$,
$\nabla^2 \sigma \nabla \sigma \Dt = \sum_i \nabla \partial_i \sigma(x_n) \partial_i \sigma(x_n) \Dt$.
This paper also needs expressions of some other terms, which will be given in \Cref{s:divkerCts}.
Then we recall the expression of the score function at $\gamma=0$.

\begin{formula} [divergence-kernel formula for score of SDEs \cite{divKer}] \label{t:div_sde}
For the SDE in $\R^M$,
\[
dx_t = F_t(x_t) dt + \sigma_t(x_t)dB_t,
\]
For any backward adapted process $\{ \alpha_t \}_{ t=0 }^T$, let $\{ \nu_t \}_{ t=0 }^T$ be the forward covector process
\begin{equation*}\begin{split}
  d \nu
  = \left( 
  (\nabla \sigma \nabla \sigma^T 
  - \nabla F^T - \alpha)\nu 
  - \nabla \div F 
  + \nabla^2 \sigma \nabla \sigma 
  + \nabla \sigma \Delta \sigma \right) dt
  - 
  (\nabla \sigma \nu^T 
  + \nabla^2 \sigma 
  + \alpha \sigma^{-1}) dB
  ,
\end{split}\end{equation*}
with initial condition $ \nu_0 = \nabla \log h_0(x_0)$,
where $h_t$ is the probability density of $x_t$, 
$\Delta \sigma$ is the Laplacian of the scalar-valued function $\sigma$.
Then
\begin{equation*}\begin{split}
  \nabla \log h_T (x_T) 
  = \E{\nu_T  \middle| x_T}.
\end{split}\end{equation*}
\end{formula}

Here, the SDE for $\nu$ is the limit of 
\begin{equation}\begin{split}
\label{e:backito}
  \Delta \nu_n 
  \approx
  \left( 
  (\nabla \sigma \nabla \sigma^T 
  - \nabla F^T -\alpha_{n+1} )\nu 
  - \nabla \div F 
  + \nabla^2 \sigma \nabla \sigma 
  + \nabla \sigma \Delta \sigma 
  \right) \Dt
  \\
  - (\nabla \sigma \nu^T 
  + \nabla^2 \sigma 
  + \alpha_{n+1} \sigma^{-1} ) \DB
\end{split}\end{equation}
as $\Dt\rightarrow 0$.
Here $\Delta \nu_n:= \nu_{n+1} - \nu_n$;
$\alpha_{n+1}$ is evaluated at step $n+1$; all other terms are evaluated at step $n$ and location $x_n$.

\subsection{Continuous-time divergence-kernel formula for linear response} 
\label{s:divkerCts}

First, we give an expression of $\delta \log | g^\gamma_{b_n*} |$, then we derive the linear response formulas.

\begin{formula}\label{t:delta_log_g}
Ignoring high-order terms which vanish as $\Dt\rightarrow 0$,
\begin{equation*}\begin{split}
  \delta \log | g^\gamma_{b_n*} | 
  \approx
  \div \delta F_n^\gamma \Dt + \nabla \delta \sigma_n^\gamma \cdot \DB_n
  - \nabla \sigma_n^T \nabla \delta \sigma_n^\gamma \Dt. 
\end{split}\end{equation*}
\end{formula}

\begin{derivation}
Substitute the expression in \Cref{e:haiyu}, we get
\begin{equation*}\begin{split}
  \delta \log | g^\gamma_{b_n*} | 
  = \frac{\delta \left| g^\gamma_{b*} \right|}{\left| g_{b*} \right|} 
  = \frac{\delta \div F^\gamma \Dt + \delta \nabla \sigma^\gamma \cdot \DB }
  {1 +  \div F \Dt + \nabla \sigma\cdot \DB } .
\end{split}\end{equation*}
Recall that $1/(1+\eps) = 1 - \eps +O(\eps^2)$, ignoring high-order terms, we get
\begin{equation*}\begin{split}
  \delta \log | g^\gamma_{b_n*} | 
  \approx 
  \left(
  \div \delta F^\gamma \Dt + \nabla \delta \sigma^\gamma \cdot \DB
  \right)
  \left(
  {1 -  \div F \Dt - \nabla \sigma\cdot \DB}
  \right)
  \\
  \approx
  \div \delta F^\gamma \Dt + \nabla \delta \sigma^\gamma \cdot \DB
  - (\nabla \delta \sigma^\gamma \cdot \DB) (\nabla \sigma \cdot \DB)
\end{split}\end{equation*}

For the last term, ignore cross terms ($\DB^i\DB^j$ with $i\neq j$) and apply $\DB^i \DB^i \approx \Dt$, 
\begin{equation*}\begin{split}
 (\nabla \delta \sigma^\gamma \cdot \DB) (\nabla \sigma \cdot \DB)
  = \sum_i \sum_j \partial_i \delta \sigma^\gamma \DB^i \partial_j \sigma \DB^j
  \approx 
  \sum_i \partial_i \delta \sigma^\gamma \partial_i \sigma \Dt
  = \nabla \sigma^T \nabla \delta \sigma^\gamma   \Dt .
\end{split}\end{equation*}
Then substitute to derive the formula.
\end{derivation}

\begin{formula} [expression of $\delta x^\gamma_n$] \label{t:deltax}
Ignoring high-order terms which vanish as $\Dt\rightarrow 0$,
\begin{equation*}\begin{split}
  \delta x^\gamma_n 
  = - g^{-1}_{b_n*} \delta g^\gamma_{b_n}(x_n)
  \approx 
  - \delta F^\gamma_n\Dt - \delta \sigma^\gamma_n \DB_n
  + \delta \sigma^\gamma_n \nabla \sigma_n \Dt.
\end{split}\end{equation*}
\end{formula}

\begin{remark*}
    This is one of the examples where some terms in the SDE setting are more explicit than discrete-time systems.
    For a discrete-time system, this term would involve solving a $M$-dimensional linear equation system.
\end{remark*}

\begin{derivation}
By \Cref{e:lifu},
\begin{equation*}\begin{split}
  g^\gamma_{b_n}(x_n) 
  := x_n + F^\gamma_n(x_n)\Dt + \sigma^\gamma_n(x_n) \DB_n,
  \\
  g_{b_n *}(x_n) 
  = I + \nabla F_n(x_n)\Dt +  \DB_n \nabla \sigma_n^T (x_n).
\end{split}\end{equation*}
Hence,
\begin{equation*}\begin{split}
  \delta g^\gamma_{b_n}(x_n) 
  = \delta F^\gamma_n(x_n)\Dt + \delta \sigma^\gamma_n(x_n) \DB_n
  \sim O(\Dt^{0.5}).
\end{split}\end{equation*}
For a small matrix $\eps\sim O(\Dt^{0.5})$, $(I+\eps)^{-1} = I - \eps + O(\Dt)$, so
\begin{equation*}\begin{split}
  g_{b_n*}^{-1}
  \approx 
  I- \nabla F_n \Dt -\DB_n \nabla \sigma_n^T.
\end{split}\end{equation*} 
Hence, ignoring $O(\Dt^{1.5})$ terms which vanish as $\Dt\rightarrow 0$,
\begin{equation*}\begin{split}
  \delta x^\gamma_n 
  = - g_{b_n*}^{-1} \delta g^\gamma_{b_n}(x_n)
  \approx 
  - (I - \nabla F_n \Dt -\DB_n \nabla \sigma^T_n)
  ( \delta F^\gamma_n\Dt + \delta \sigma^\gamma_n \DB_n)
  \\
  \approx 
  - \delta F^\gamma_n\Dt - \delta \sigma^\gamma_n \DB_n
  + \DB_n \nabla \sigma^T_n \delta \sigma^\gamma_n \DB_n
\end{split}\end{equation*} 
Note that $ \nabla \sigma^T_n \DB_n $ is a number, and ignoring the cross terms, we get
\begin{equation*}\begin{split}
  \DB_n (\nabla \sigma^T_n \DB_n)
  = \DB_n (\DB_n^T \nabla \sigma_n)
  \approx \nabla \sigma_n \Dt.
\end{split}\end{equation*}
Then substitute to prove the lemma.
\end{derivation}

\begin{formula} \label{t:yong}
Ignoring high-order terms which vanish as $\Dt\rightarrow 0$,
let $\Delta \sigma$ denote the Laplacian of $\sigma$,
then
\begin{equation*}\begin{split}
  \nabla \log \left|  g_{b_n*} \right|
  \cdot \delta x^\gamma_n 
  \approx 
  - \delta \sigma^\gamma_n \Delta \sigma_n \Dt.
\end{split}\end{equation*}
\end{formula}

\begin{derivation}
By \Cref{e:haiyu} and \Cref{t:deltax},
\begin{equation*}\begin{split}
  \nabla \log \left|  g_{b_n*} \right|
  \approx 
  \nabla \div F_n(x_n) \Dt 
  + \nabla^2 \sigma_n(x_n) \DB_n
  - \nabla^2 \sigma_n \nabla \sigma_n (x_n) \Dt,
  \\
  \delta x^\gamma_n 
  = - g^{-1}_{b_n*} \delta g^\gamma_{b_n}(x_n)
  \approx 
  - \delta F^\gamma_n\Dt - \delta \sigma^\gamma_n \DB_n
  + \delta \sigma^\gamma_n \nabla \sigma^T_n \Dt.  
\end{split}\end{equation*}
In the product, the only term smaller than or equal to $O(\Dt)$ is 
\begin{equation*}\begin{split}
  \nabla \log \left|  g_{b_n*} \right|
  \cdot \delta x^\gamma_n 
  \approx 
  \nabla^2 \sigma_n(x_n) \DB_n
  \cdot 
  - \delta \sigma^\gamma_n \DB_n
\end{split}\end{equation*}
Ignore the cross term, we get
\begin{equation*}\begin{split}
  \nabla^2 \sigma \DB \cdot \DB
  = \sum_{i,j} \partial_i\partial_j \sigma \DB^j \DB^i 
  \approx \sum_i \partial_i^2 \sigma \Dt
  =: \Delta \sigma \Dt ,
\end{split}\end{equation*}
Then substitute to prove the lemma.
\end{derivation}

With these, we can formally pass \Cref{t:divkernstep} to the continuous-time limit to derive the main formula of this paper.

\begin{derivation}{(of \Cref{t:divker_sde})}
Substitute \Cref{t:delta_log_g,t:deltax,t:yong} into \Cref{t:divkernstep}, we get
\begin{equation} \begin{split} \label{e:siye}
  \delta \log h^\gamma_N (x_N) 
  = 
  \E{
  \delta \log h^\gamma_0 
  + \sum_{n=0}^{N-1}
  \nu_n \cdot \delta x^\gamma_n 
  - \nabla \log |g_{b_n*}| \cdot \delta x^\gamma_n 
  - \delta \log | g^\gamma_{b_n*} | 
  \middle| x_N}
  \\
  \approx
  \mathbb{E}\left[
  \delta \log h^\gamma_0 
  + \sum_{n=0}^{N-1}
  \nu_n \cdot (- \delta F^\gamma_n\Dt - \delta \sigma^\gamma_n \DB_n + \delta \sigma^\gamma_n \nabla \sigma_n \Dt)
  \right.
  \\
  \left.
  + \delta \sigma^\gamma_n \Delta \sigma_n \Dt
  - \div \delta F_n^\gamma \Dt 
  - \nabla \delta \sigma_n^\gamma \cdot \DB_n
  + \nabla \sigma_n^T \nabla \delta \sigma_n^\gamma \Dt
  \middle| x_N \right].
\end{split} \end{equation}
Then formally pass to the limit $\Dt\rightarrow 0$.
\end{derivation}

We also formally pass to the infinite-time limit of stationary distributions $h^\gamma$.
We need to further assume that $F_t$ and $\sigma_t$ are stationary, and that there is $T'$ such that $\alpha_t$ only depends on $x_t\sim x_{T'}$, so $(\alpha, \nu, x)$ are jointly stationary.
\begin{formula} [ergodic divergence-kernel linear response formula] \label{t:ergodic}
For any $C^2$ observable function $\Phi:\R^M\rightarrow R$, let $h^\gamma$ be the stationary measure under parameter $\gamma$, then the linear response of the averaged observable can be expressed almost surely on one path:
\begin{equation*} \begin{split}
  \delta \int \Phi(x) h^\gamma(x) dx
  \overset{a.s.}{=}
  \lim_{W\rightarrow\infty}
  \lim_{T\rightarrow\infty}
  \frac 1T
  \int_{t=0}^{T}
  \left(
  \int_{\tau=0}^{W}
  (\Phi(x_{t+\tau})-\Phi_{avg}) d\tau
  \right)
  (\delta \sigma_t^\gamma \nabla \sigma_t^T \nu 
  \\
  - \delta F_t^{\gamma T} \nu
  - \div \delta F_t^\gamma 
  + \delta \sigma_t^\gamma \Delta \sigma_t 
  + \nabla \sigma_t^T \nabla \delta \sigma_t^\gamma ) dt
  - 
  (\delta \sigma_t^\gamma \nu^T
  + \nabla \delta \sigma_t^{\gamma T}) dB_t,
  \\
  \textnormal{where} \quad 
  \Phi_{avg}:= \int \Phi(x) h(x)
  \overset{a.s.}{=}
  \lim_{T\rightarrow\infty} \frac 1T
  \int_{t=0}^{T} \Phi(x_{t}) dt.
  \\ 
\end{split} \end{equation*}
\end{formula}

The discrete-time expression is
\begin{equation*}\begin{split}
  \delta \int \Phi(x) h^\gamma(x) dx
  \overset{a.s.}{=}
  \lim_{W\rightarrow\infty}
  \lim_{N\rightarrow\infty}
  \frac 1N
  \sum_{n=0}^{N}
  \left(
  \sum_{m=1}^{W}
  (\Phi(x_{n+m})-\Phi_{avg})
  \right)
  \nu_n \cdot \delta \sigma^\gamma_n \nabla \sigma_n \Dt
  \\ 
  -  \nu_n \cdot( \delta F^\gamma_n\Dt + \delta \sigma^\gamma_n \DB_n )
  + \delta \sigma^\gamma_n \Delta \sigma_n \Dt
  - \div \delta F_n^\gamma \Dt 
  - \nabla \delta \sigma_n^\gamma \cdot \DB_n
  + \nabla \sigma_n^T \nabla \delta \sigma_n^\gamma \Dt.
\end{split}\end{equation*}

\subsection{Discussions}
\label{s:howtouse}

The discussions in \cite{divKer} are still valid.
Roughly speaking, we want $\alpha$ larger than the largest contraction rate.
Also, the divergence-kernel method does not work well in highly contracting systems with little to no noise, the case where the path-perturbation method is good.
We may need to use all three methods (path, kernel, and divergence) together to obtain an ultimate solution, as proposed in \cite{Ni_kd}.

The divergence-kernel method contains second derivatives.
To speed up our code, note that we do not really need the matrices such as $\nabla^2 \sigma$ and $\nabla F^T$, we only need their products with certain vectors.
Such matrix-vector products can be implemented much more efficiently in high dimensions (such as the jvp function in JAX) than forming the matrix, especially with sparse structures.

Compared to the path-kernel method in \cite{dud,apk}, the divergence-kernel method has the following advantages:
\begin{enumerate}
  \item It involves only forward processes, so the memory cost is much smaller.
  \item Its cost is almost independent of \textit{both} the number of observables \textit{and} the parameters.
  In contrast, the cost of the path-kernel method is independent of \textit{either} the number of observables \textit{or} the parameters.
\end{enumerate}
It also has some disadvantages:
\begin{enumerate}
  \item Its computational complexity per step is higher than that for the path-kernel method, although still on the same order as that for the path-kernel method. This is because it involves the second-order derivative.
  \item For dynamical systems generated from discretized PDEs, typically $|\lambda_{min}|>|\lambda_{max}|$, where $\lambda_{max}$ and $\lambda_{min}$ are the largest and the smallest (most negative) Lyapunov exponents; hence, the path-kernel method requires a smaller $\alpha$, so the estimator has smaller size and the average converges faster.
\end{enumerate}

\section{Numerical examples for linear responses}
\label{s:numeric}

\subsection{1-dimensional process with multiplicative noise}
\label{s:eg_1dim}

We use \Cref{t:divker_sde} to compute the linear response of the one-dimensional process with multiplicative noise. The SDE is
\begin{equation*}\begin{split}
  dx = \beta_t (\gamma-x) dt + \sqrt{\beta_t} \left(0.5 + \exp {(\gamma-x)^2} \right) dB,
  \quad \textnormal{where} \quad 
  \beta_t:= 1 + 3t.
\end{split}\end{equation*}
The initial distribution is $x_0\sim N(0,1)$. 
Set $T = 1$, $\Dt=0.01$, $\alpha=10$.
Note that the problem is symmetric for $\gamma$.

We are only interested in terminal states $x_T$ in the interval $[-2,2]$, which is further partitioned into $10$ subintervals.
We simulate many sample paths, then count the number of paths whose $x_T$ fall into a specific subinterval, to obtain the empirical distribution of $x_T$ on subintervals.
Then we compute the linear response $\delta \log h_T$ by averaging the divergence-kernel formula over these paths, and see if our results reflect the trend of $\log h_T$ from the empirical distribution.
We also use the average of $\nu_T$ to compute the scores $\nabla \log h_T$. 
Each subinterval contains at least 180 samples.
The $\delta \log h_T$ and $\nabla \log h_T$ computed for different $\gamma$ and $x$ are shown in \Cref{f:1d_ker}.
Both correctly reflect the trend of $\log h_T$.

\begin{figure}[ht] \centering
  \includegraphics[width=0.44\textwidth]{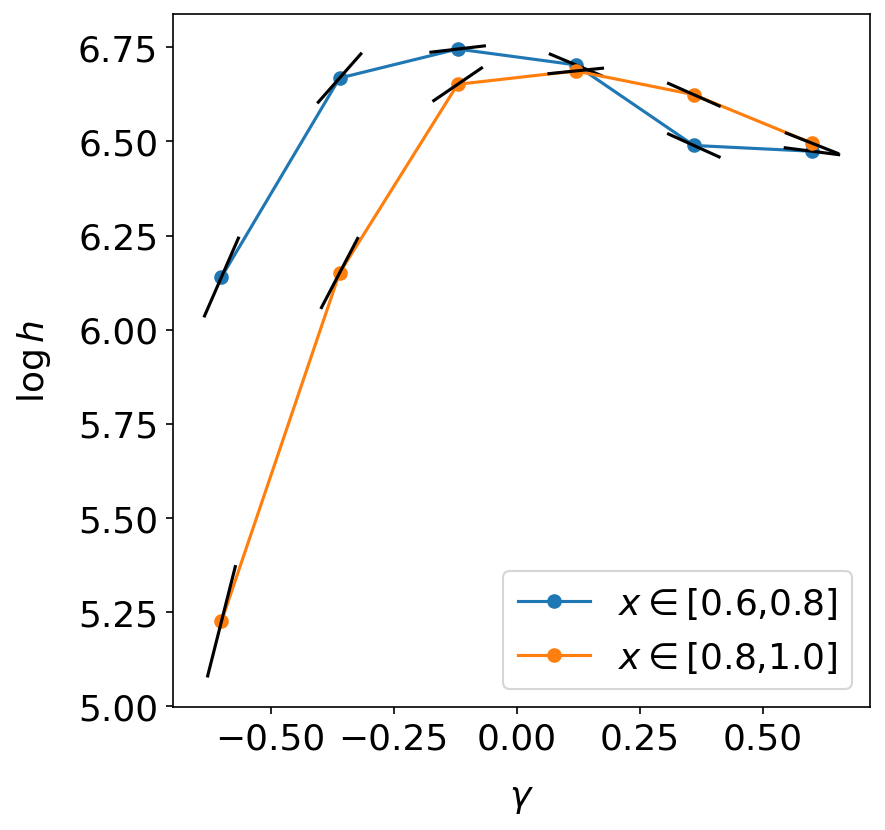}
  \includegraphics[width=0.45\textwidth]{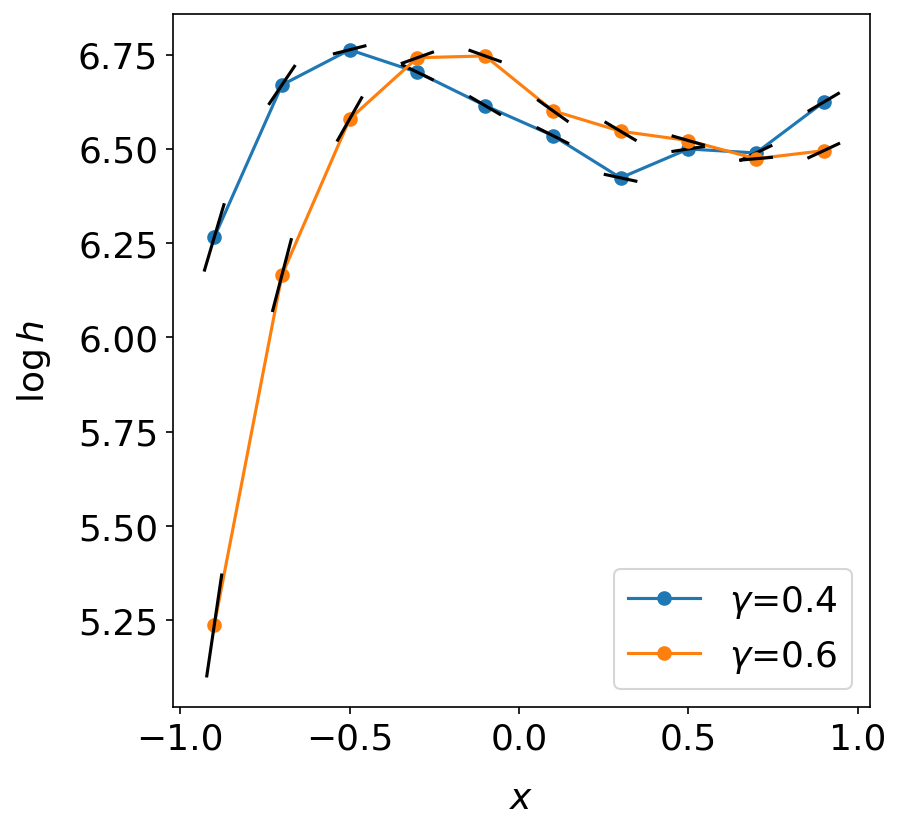}
  \caption{
  Divergence-kernel method for the linear responses and scores of the 1-dimensional SDE with multiplicative noise, $T=1$.
  The dots are $\log h_T$.
  Each short line is a linear response or score computed by the divergence-kernel algorithm, averaged on paths whose terminal states $x_T$ fall into the same subinterval.
  Left: linear responses $\delta \log h_T$.
  Right: scores $\nabla \log h_T$ .
  }
  \label{f:1d_ker}
\end{figure}

\subsection{40-dimensional Lorenz 96 system with multiplicative noise}
\label{s:eg_lorenz}

We use \Cref{t:ergodic} to compute the linear response of the stationary density of the Lorenz 96 model \cite{Lorenz96}.
The dimension of the system is $M=40$.
The SDE is
\begin{eqnarray*}
  d x^i
  = \left( \left(x^{i+1}-x^{i-2}\right) x^{i-1} - x^i + 8 +\gamma - 0.01 (x^i)^2 \right) dt + \left( 0.5+\exp{|x-\gamma|^2} \right) \, dB^i,
  \\
  \textnormal{where} \quad
  i=0, \ldots, M-1;
  \quad
  x_0 \sim \cN (0, I).
\end{eqnarray*}
Here $i$ labels different directions in $\R^M $, and We use cyclic indexing modulo $M$.
We added noise and the $ - 0.01 (x^i)^2$ term, which prevents the noise from carrying us to infinitely far away.
We choose the observable function to be
\[
  \Phi(x) := \frac 1M |x|^2.
\]

We use the Euler integration scheme with $\Dt=0.002$.
A typical orbit is in \Cref{f:lorenz}.
Then we use the ergodic linear response formula \Cref{t:ergodic} to compute the linear response of the observable averaged on the stationary measure.
The averaged observable $\Phi_{avg}$ and the linear response $\delta \int \Phi h^\gamma$ are averaged over 7 orbits of length $T=400$.
The decorrelation time is $W=1.5$.
The results shown in \Cref{f:lorenz}; the linear response computed using the divergence-kernel method is correct.

\begin{figure}[ht] \centering
\includegraphics[height=0.43\textwidth]{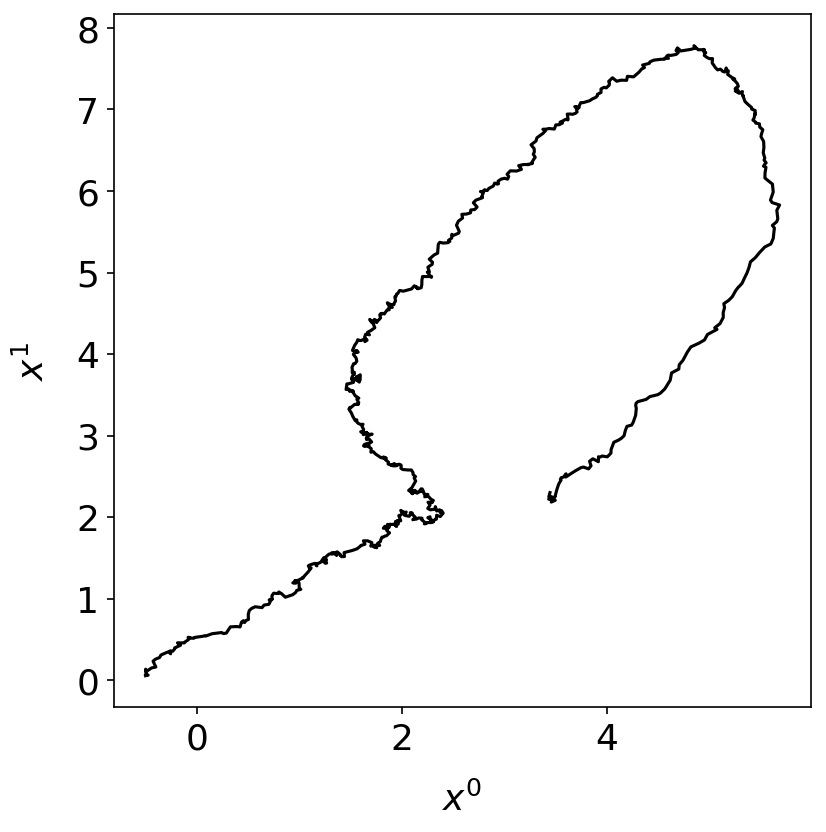} \hfill
\includegraphics[height=0.43\textwidth]{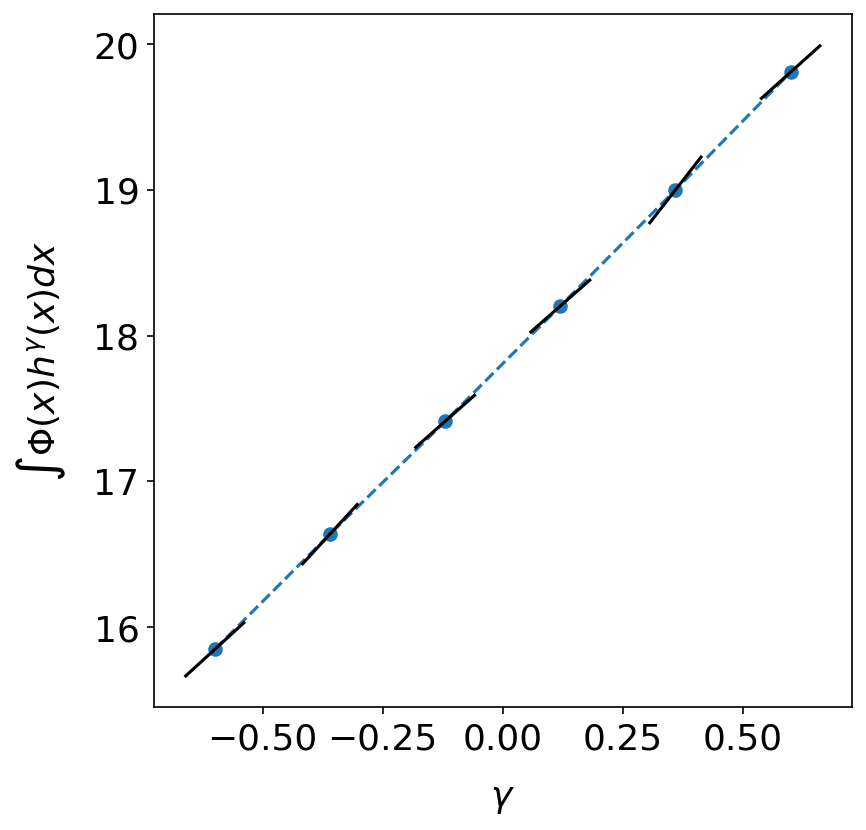}
  \caption{
  Left: plot of $x^0_t, x^1_t$ from a typical orbit of time length $1.5$. 
  Right: linear response computed by divergence-kernel method.
  }
  \label{f:lorenz}
\end{figure}

\section{Parametric SDE generative model trained by the divergence-kernel method (DK-SDE)}
\label{s:diffmodel}

We consider the problem in which we are given $N_{\mathrm{data}}$ many samples $y_k\in \R^{M}$, whose true distribution is unknown. 
Then we use the divergence-kernel method to compute the gradient, and use gradient descent to train a parametric generative SDE model which maps the Gaussian distribution to match the data.

\subsection{Problem and algorithm description}

We describe the problem and our model, then we define loss function and explain how to compute its gradient, and how to update the parameters by gradient descent.

First, we describe how to generate samples.
Note that the information in this paragraph will not be available when we train the model.
The samples are generated by the SDE
\begin{equation*}\begin{split}
  dx_t = F^{\gamma_{true}}_t(x_t) dt + \sigma^{\gamma_{true}}_t(x_t)dB_t,
\end{split}\end{equation*}
where the true parameters $\gamma_{true}:=[\gamma_{true} ^0,\ldots, \gamma_{true}^{N_\gamma-1}]$.
The samples are $y:=x_T$ obtained at the end of independent paths.
Denote the target distribution (the one we try to match) as $P$; when the target is empirical distribution prescribed by a set of data, $P$ is a sum of Dirac delta functions.
Our framework can work on more general data not necessarily generated by an SDE, but here we restrict ourselves to this simple case, so that after optimization, we can check our result with the true parameters.

We want to fit several unknown parameters $\gamma=[\gamma ^0,\ldots, \gamma^{N_\gamma-1}]$ in the generative SDE model such that its marginal distribution matches the data.
The prototype of the SDE is
\begin{equation*}\begin{split}
  dx_t = F^{\gamma}_t(x_t) dt + \sigma^{\gamma}_t(x_t)dB_t.
\end{split}\end{equation*}
Denote the density at $T$ by $h^\gamma$ and measure by $H^\gamma$ -- this is the distribution that we want to optimize.

In the current paper, we use simple models, such as one based on the Lorenz 96 system.
These simple models can be interpreted as the result of the \textit{prior} knowledge we had about how the data was generated. 
Such prior knowledge can be naturally incorporated into the SDE model.
In the future, we will test on problems with less prior knowledge, so we will need more flexible models to express $F$ and $\sigma$, such as neural networks.

Also, note that the gradient used in the optimization is \textit{not} computed by backpropagation; rather, it only involves a covector process that runs forward in time and uses much less memory.
We also remark that we can choose to optimize the stationary distribution using \Cref{t:ergodic}, which may be more efficient for some cases.

\subsection{KL divergence as loss and gradient descent}

Our training loss is the KL divergence, but we never need to evaluate the KL-divergence $D_{\mathrm{KL}}(P\|H^\gamma)$; we only need its parameter-gradient.
When $P$ admits a density $p$ (w.r.t.\ Lebesgue measure), the KL-divergence is
\[
D_{\mathrm{KL}}(P\| H^\gamma)=\int p(y) \log\frac{p(y)}{h^\gamma(y)}\, dy
= \underbrace{\int p(y)\log p(y)\,dy}_{\text{independent of }\gamma}
-\int p(y)\log h^\gamma(y)\,dy.
\]
Hence, its parameter-derivative is
\[
  \delta D_{\mathrm{KL}}(P\| H^\gamma)
  = - \int_{\R^M}
  p(y) \delta \log h^\gamma (y) \,dy.
\]
If $P$ is given only through samples $\{y_k\}_{k=1}^{N_{\mathrm{data}}}\sim P$ (or is strictly singular), we optimize the same $\gamma$-dependent term (cross-entropy), namely
\[
\mathcal L(\gamma):=-\frac1{N_{\mathrm{data}}}\sum_{k=1}^{N_{\mathrm{data}}}\log h^\gamma(y_k),
\qquad
\delta \mathcal L(\gamma)= -\frac1{N_{\mathrm{data}}}\sum_{k=1}^{N_{\mathrm{data}}} \delta \log h^\gamma(y_k).
\]
We can compute the cross-entropy, but there is little benefit in doing so here, and it would require a few additional steps.
In this paper, we judge the accuracy of the trained model by the discrepancy between the trained parameters and the ground-truth parameters; this criterion is complementary to the KL/cross-entropy.

Then we need to decide how to compute $\delta \log h^\gamma$ at the prescribed locations $y_k$.
It is difficult to obtain paths ending exactly at $y_k$, so we need an approximation.
For each $\gamma$ during optimization, we simulate $L$-many paths by this SDE, and we label each path by $l$, in particular, the terminal location is denoted by $x_{T,l}$. 
We approximate $y_k$ by its $N_{\mathrm{neighbor}}$-nearest neighbors in $\{x_{T,l}\}_{l=1}^L$, and therefore $\delta \log h^\gamma (y_k)$ is approximated by the average of the divergence-kernel expression in \Cref{t:divker_sde} over those $N_{\mathrm{neighbor}}$ paths.
Hence,
\begin{equation}\begin{split}
\label{e:shui}
  \delta D_{\mathrm{KL}}(P\| H^\gamma)
  \approx
  \hat \delta D_{\mathrm{KL}}(P\| H^\gamma)
  := -
  \frac 1{N_{\mathrm{data}}} \sum_{k=1}^{N_{\mathrm{data}}}
  \frac 1{N_{\mathrm{neighbor}}} \sum_{x_{T,l} \in N(y_k)} \cT (x_{T,l}),
\end{split}\end{equation}
where $N(y_k)$ is the neighbors of $y_k$, and $\cT(x_{T,l})$ is the expression in \Cref{t:divker_sde} on the path ending at $x_{T,l}$.

After approximately computing the parameter-derivative, we update the parameter by gradient descent,
\begin{equation*}\begin{split}
  \gamma_{n+1}
  =
  \gamma_n - \eta \hat \delta D_{\mathrm{KL}}(P\| H^\gamma).
\end{split}\end{equation*}
For simplicity, we keep $\eta$ constant, but we can also change it after each update of $\gamma$.

We can give a direct interpretation of the overall optimization strategy.
Before each update, we simulate $L$ many sample paths, then we decide which of these paths' terminal state is closest to the data (this is the effect of choosing $N_{\mathrm{neighbor}}$-nearest neighbors).
On the other hand, the divergence-kernel linear response formula tells us how changing parameters affects the marginal log-densities at those terminal states.
Combining both facts, we can now update parameters in the direction which increases the marginal log-densities on those terminal states close to the data.
Hence, overall, each parameter update moves the marginal distribution of the SDE towards the data.

This interpretation also suggests that the best choice of $N_{\mathrm{neighbor}}$ is often $1$ or a small number.
Indeed, increasing $N_{\mathrm{neighbor}}$ replaces the ideal conditional expectation $\mathbb E[\cdot\mid x_T=y_k]$ by a local average over nearby terminal states.
While this can reduce Monte-Carlo variance, it also introduces a bias: points in high-density regions of the sample cloud tend to be selected as neighbors more frequently than points near the periphery.
As a result, when $N_{\mathrm{neighbor}}$ is too large, the update may overweight central samples and encourage an artificial contraction of the learned distribution toward its center, even when the parameters are already close to optimal.
For example, in the idealized situation where the simulated terminals match the data exactly, taking $N_{\mathrm{neighbor}}=1$ yields an unbiased 1-to-1 neighbor assignment, whereas $N_{\mathrm{neighbor}}>1$ mixes in off-target terminals and reintroduces geometric bias toward locally denser regions.
Therefore we typically take $N_{\mathrm{neighbor}}=1$ (or a small value) unless additional smoothing is needed for stability.

We stress that the nearest-neighbor step is merely a simple numerical device for approximating the conditional expectation appearing in the divergence-kernel formula.
One could replace it by other local averaging schemes without changing the underlying update principle.
Therefore, the more substantive question is not the optimality of nearest-neighbor, but whether the resulting update direction admits the correct interpretation: namely, that it increases the model marginal log-density in regions supported by the data.


\subsection{A 1D problem with 2 parameters}

First, we try a 1-dimensional problem.
The SDE model is
\begin{equation*}\begin{split}
  d x
  = \left(\gamma^0-x \right) dt + \left(0.5+\exp{-(x-\gamma^1)^2}\right) dB,
  \quad \textnormal{where} \quad 
  x_0 \sim \cN (0, 1).
\end{split}\end{equation*}
The true parameter used to generate the data is $\gamma_{true}=[0,0]$.
The initial parameter in the optimization is $\gamma_0=[5,1]$.
The number of data points is $N_{\mathrm{data}}=200$.
The number of samples generated for each update of $\gamma$ is $L=200$.
The stepsize (or learning rate) is $\eta=1$.
The number of neighbors used to approximate each $y_k$ is $N_{\mathrm{neighbor}}=5$; we will test the effect of $N_{\mathrm{neighbor}}$ in the next example with higher dimension $M$.

Now we decide $\alpha$. 
For simplicity, we let it be a constant. 
To decide its value, recall that $\alpha$ is used to temper the growth of the covector process $\nu$.
In view of \Cref{t:divker_sde}, the homogeneous part of $\nu$ grows like
\begin{equation*}\begin{split}
  d \nu 
  = (\nabla \sigma \nabla \sigma^T 
  - \nabla F^T - \alpha)\nu dt + \cdots
\end{split}\end{equation*}
If $\alpha$ is too small then $\nu$ would grow exponentially fast; we do not want this.
For this 1D problem, we can directly compute these terms.
We want the homogeneous part of $\nu$ to decay, which can be achieved if
\begin{equation*}\begin{split}
  \nabla \sigma \nabla \sigma^T -
  \nabla F^T 
  = 4 \exp{-2(x-\gamma^1)^2}(x-\gamma^1)^2 + 1
  \le 2/e + 1 \le 3 \le \alpha
\end{split}\end{equation*}
So we can just set $\alpha$ to be somewhat larger than $3$, for example, we set $\alpha=5$.
For more complicated situations, we can just try a few $\alpha$ and pick the one whose corresponding $\nu$ is moderate.

We start from $\gamma_0$ and use gradient descent to minimize the KL-divergence.
The histograms of data $\{y_k\}$ and samples $\{x_{T,l}\}$ for $\gamma_0$ and $\gamma_{10}$, the parameter after 10 updates, are shown in \Cref{f:gm1d}.
As we can see, the two distributions get much closer during the optimization.
More specifically, the error $\gamma-\gamma_{true}$ is reduced from $[5,1]$ to $[-0.05,0.09]$.

\begin{figure}[ht] \centering
\includegraphics[height=0.32\textwidth]{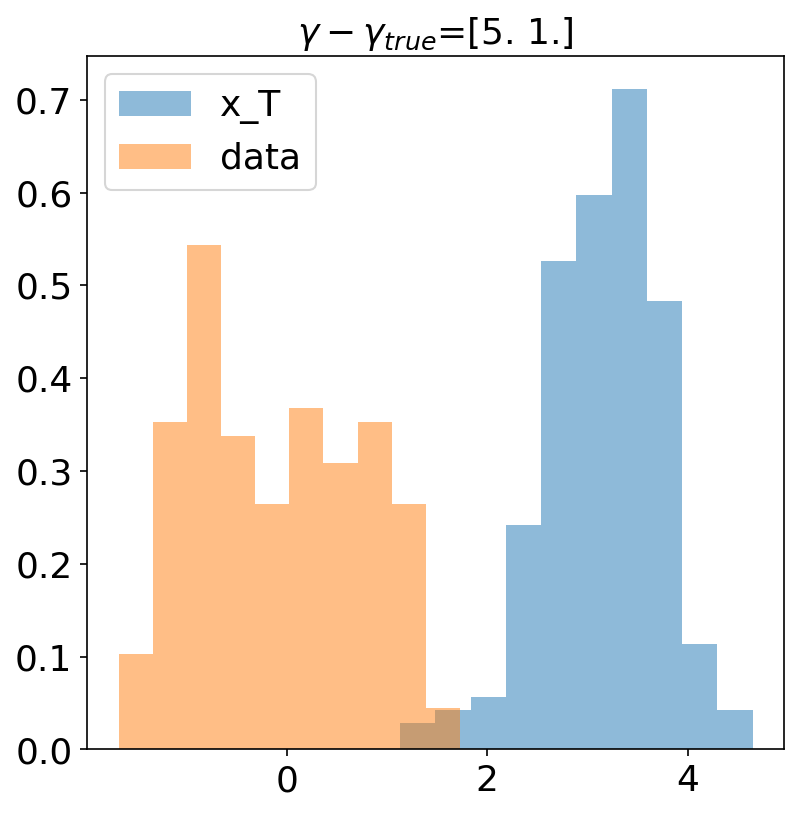} \hfill
\includegraphics[height=0.32\textwidth]{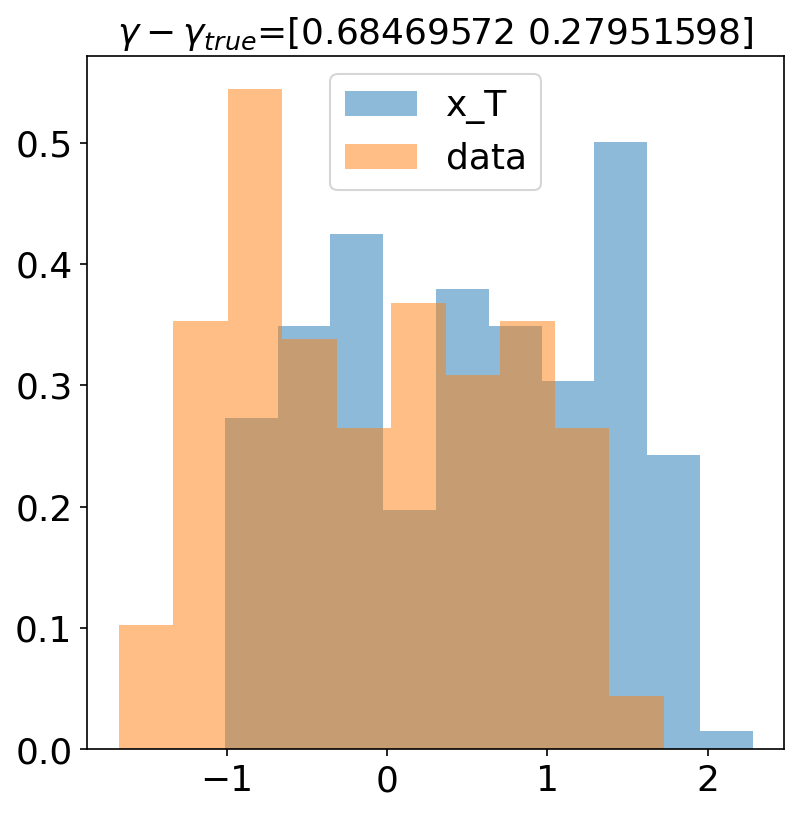} \hfill
\includegraphics[height=0.32\textwidth]{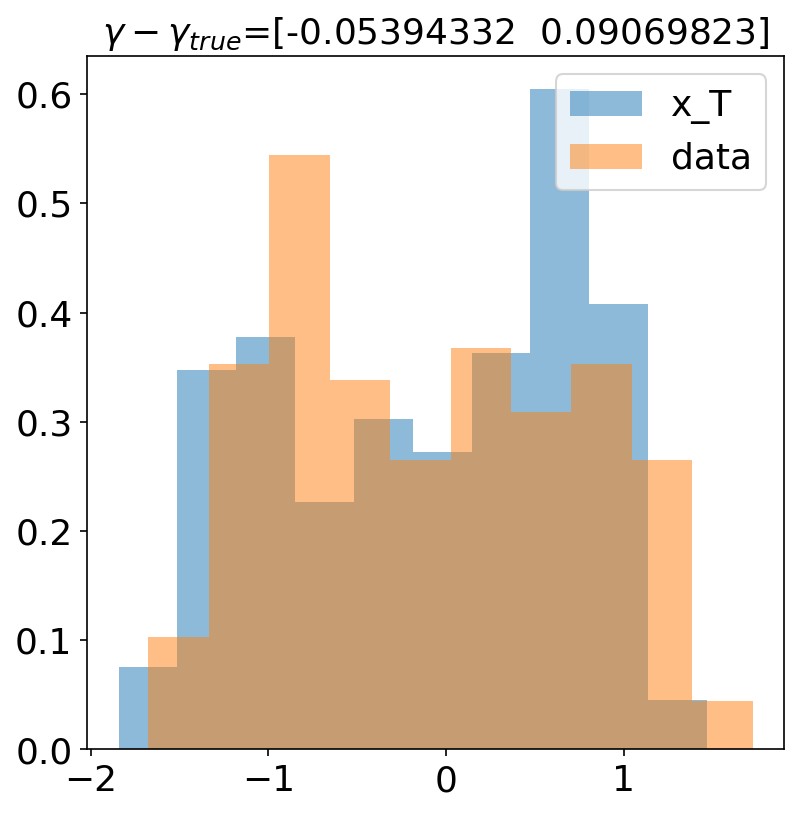}
  \caption{
  Histograms of data $\{y_k\}$ and samples $\{x_{T,l}\}$.
  Left to right: plots at $\gamma_0$, $\gamma_{2}$, and $\gamma_{10}$.
  }
  \label{f:gm1d}
\end{figure}

\subsection{A 5D problem with 6 parameters}
\label{s:generative_model_5d}

We try a 5-dimensional Lorenz 96 system.
The SDE model is
\begin{eqnarray*}
  d x^i
  = \left( \left(x^{i+1}-x^{i-2}\right) x^{i-1} - x^i + \gamma^i - 0.01 (x^i)^2 \right) dt 
  + \left(0.5+\exp{-(x-\gamma^{M}\one)^2} \right) dB^i,
  \\
  \textnormal{where} \quad
  i=0, \ldots, M-1;
  \quad
  M = 5;
  \quad
  \one=[1,1,1,1,1];
  \quad
  x_0 \sim \cN (0, I).
\end{eqnarray*}
The true parameter used to generate the data is $\gamma_{true}=[5,6,7,8,9,2]$.
The initial parameter in the optimization is $\gamma_0=[0,0,0,0,0,0]$.
The number of data points is $N_{\mathrm{data}}=200$.
The number of samples generated for each update of $\gamma$ is $L=200$.
The stepsize (or learning rate) is $\eta=1$.
The number of neighbors used to approximate each $y_k$ is $N_{\mathrm{neighbor}}=5$.

We start from $\gamma_0$ and use gradient descent to minimize the KL-divergence.
The history of $|\gamma_n-\gamma_{true}|^2/N_\gamma$ for $n=0$ to $50$ is shown in \Cref{f:gmlorenz}.
As we can see, $\gamma_n$ gets closer to $\gamma_{true}$ during optimization, with $\gamma_{48}$ achieving the minimum distance.
More specifically,
\begin{equation*}\begin{split}
 \gamma_0-\gamma_{true}
 = [-5, -6, -7, -8, -9, -2];
 \\
 \gamma_{48}-\gamma_{true}
 = [-0.33, 0.40, -0.26, -0.20, -0.24, -0.36].
\end{split}\end{equation*}

\begin{figure}[ht] \centering
\includegraphics[width=0.4\textwidth]{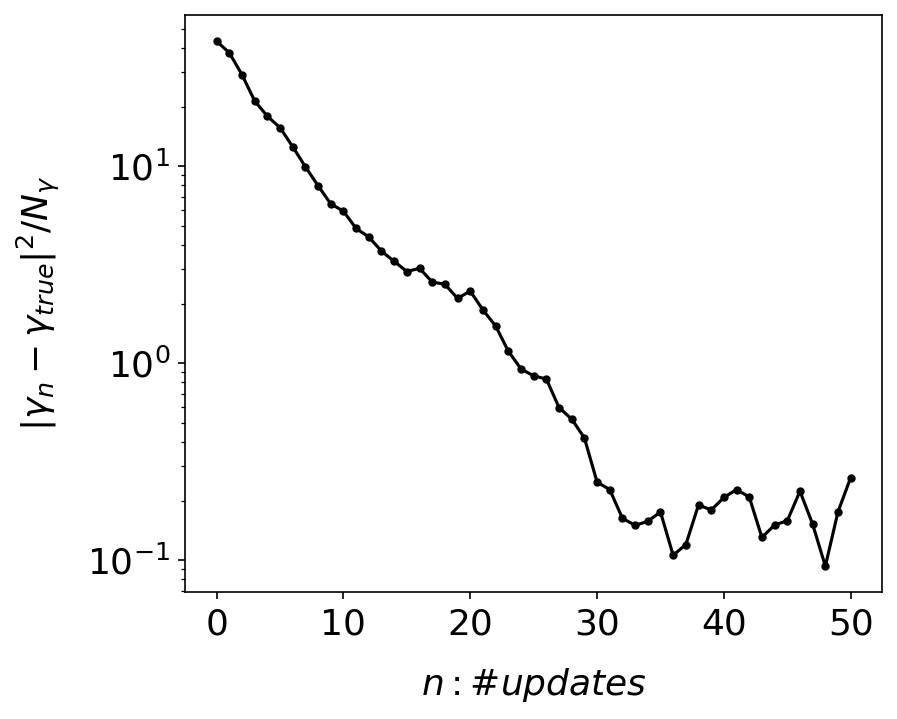} \hfill
  \caption{
  History of $|\gamma_n-\gamma_{true}|^2/N_\gamma$ for the 5D DK-SDE model.
  }
  \label{f:gmlorenz}
\end{figure}

\subsection{A 20D problem with 21 parameters}

We try a 20-dimensional system.
More importantly, we run some convergence test against several hyper-parameters and record computer time.
The SDE model is similar to the one in \Cref{s:generative_model_5d}, except for that now we set dimension $M=20$ and $N_\gamma=21$.
We set \(\gamma_{\mathrm{true}}\in\mathbb{R}^N_\gamma\) to ramp linearly from \(5\) to \(9\) over the first ten components and then ramp linearly back from \(9\) to \(5\) over the next ten components; additionally, we set the last entry \(\gamma_{\mathrm{true},\,N_\gamma-1}=1\).
The initial guess of all parameters are still zeros.
The number of data points is $N_{\mathrm{data}}=200$.

We vary several hyperparameters; when sweeping one parameter, we keep the others fixed at their default values.
In \Cref{f:gm20d_history} we plot the convergence history of \( |\gamma_n-\gamma_{\mathrm{true}}|^2/N_\gamma \) for \(n=0,\dots,1000\) under different choices of hyperparameters.
Specifically, we sweep \(\eta\), \(N_{\mathrm{neighbor}}\), and \(L\), while holding the remaining two at the defaults \(\eta=0.1\), \(N_{\mathrm{neighbor}}=1\), and \(L=3000\).
Smaller \(\eta\) leads to slower convergence but a better final error; in future work, it may therefore be beneficial to decrease \(\eta\) according to a schedule.
Using \(N_{\mathrm{neighbor}}=1\) introduces some oscillations but yields the best overall result, which is encouraging since this choice does not depend on the dimension \(M\).
Finally, increasing the sample size \(L\) improves the results, as expected.

\begin{figure}[ht] \centering
  \makebox[\textwidth][c]{
    \begin{minipage}{1.2\textwidth}
    \centering
    \includegraphics[height=0.32\textwidth]{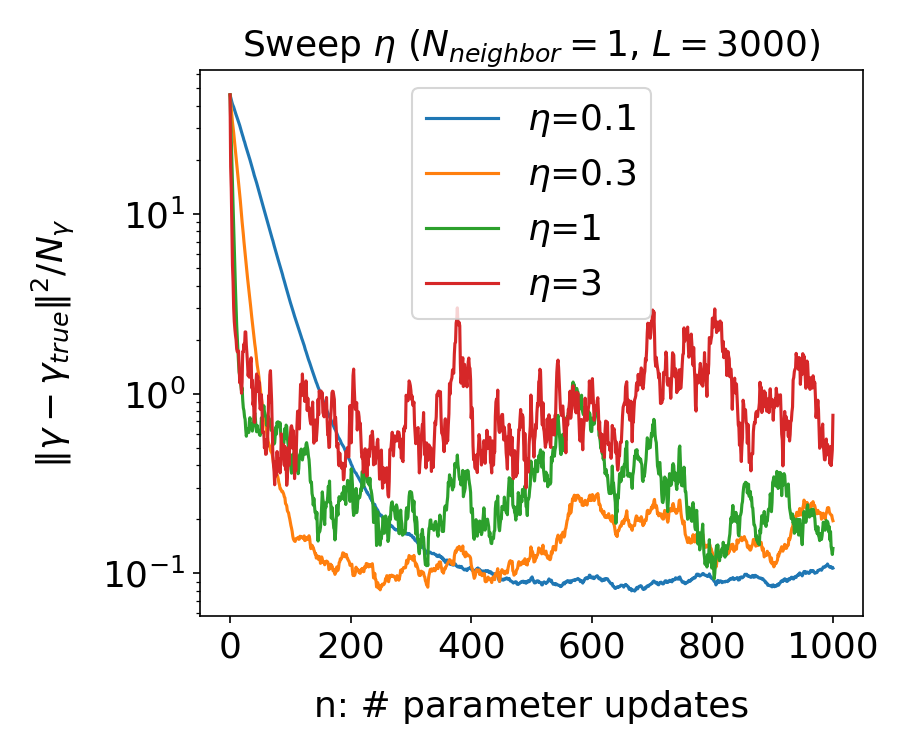}\hfill
    \includegraphics[height=0.32\textwidth,trim={32mm 0 0 0},clip]{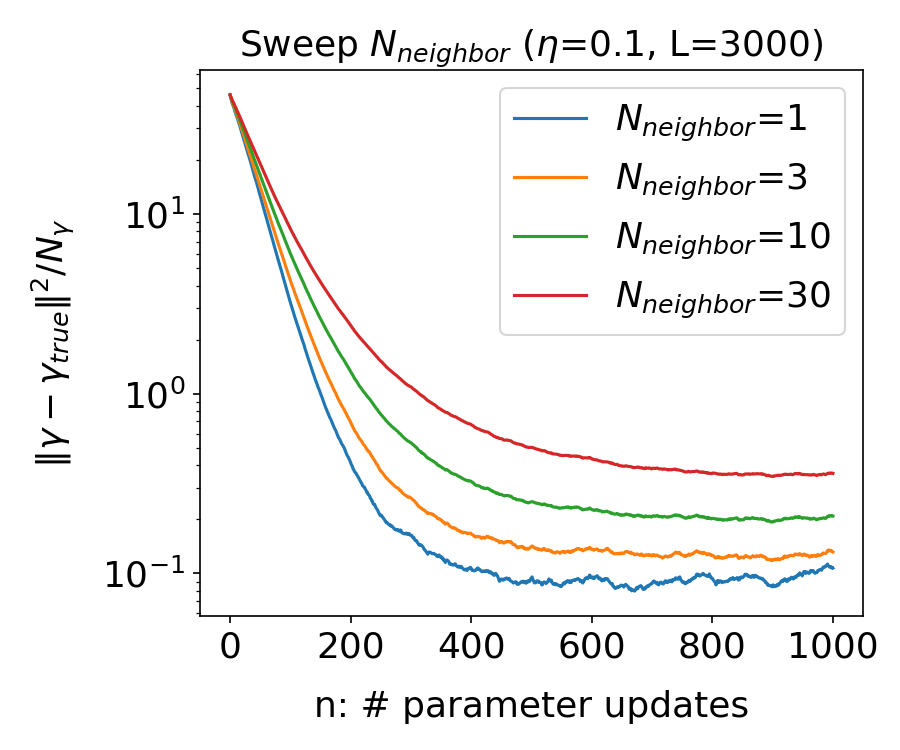}\hfill
    \includegraphics[height=0.32\textwidth,trim={32mm 0 0 0},clip]{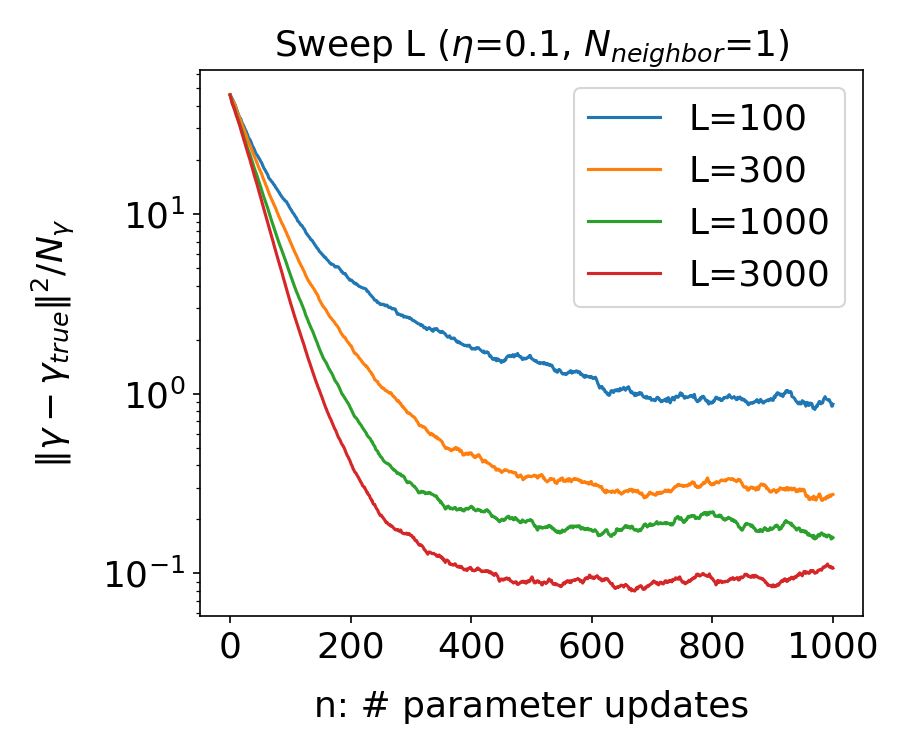}
    \end{minipage}
  }
  \caption{
    Convergence history for the 20D DK-SDE model.
    From left to right: sweep over learning rate $\eta$, $N_{\mathrm{neighbor}}$, and number of samples $L$.
  }
  \label{f:gm20d_history}
\end{figure}

Then we plot the data and generated samples under the default hyperparameters in \Cref{f:gm20d_scatter}.
The wall time for 1000 updates on an RTX 5060 GPU is about 201 seconds.
As intended, the sample distribution moves closer to the data distribution during optimization.
More specifically, the scatter plots indicate good agreement at least in the first and second moments, which is a natural consequence of distributions getting closer.

\begin{figure}[ht] \centering
\makebox[\textwidth][c]{
    \begin{minipage}{1.1\textwidth}
    \centering
    \includegraphics[height=0.37\textwidth]{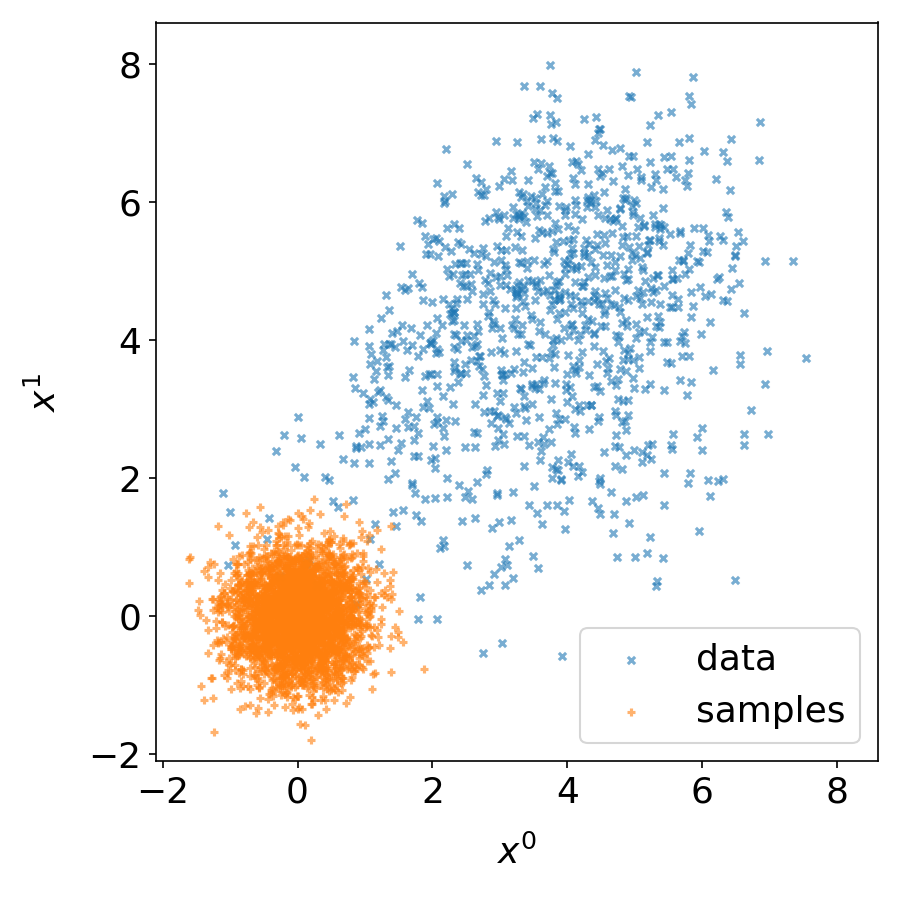}\hfill
    \includegraphics[height=0.37\textwidth,trim={25mm 0 0 0},clip]{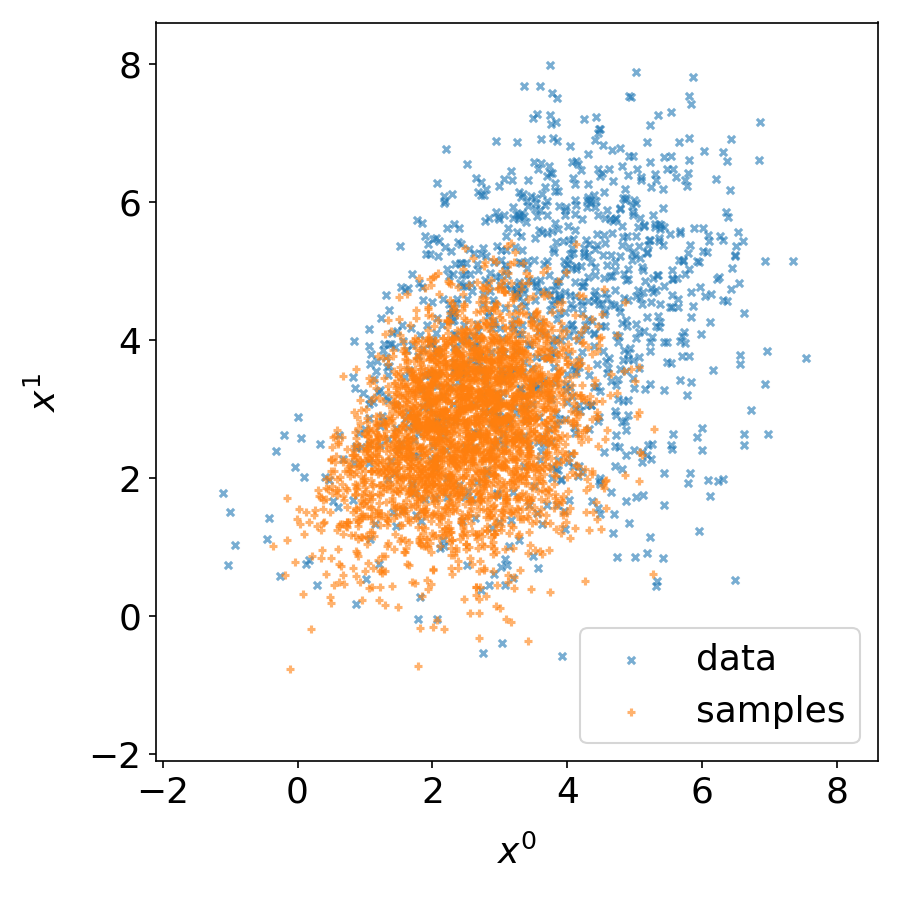}\hfill
    \includegraphics[height=0.37\textwidth,trim={25mm 0 0 0},clip]{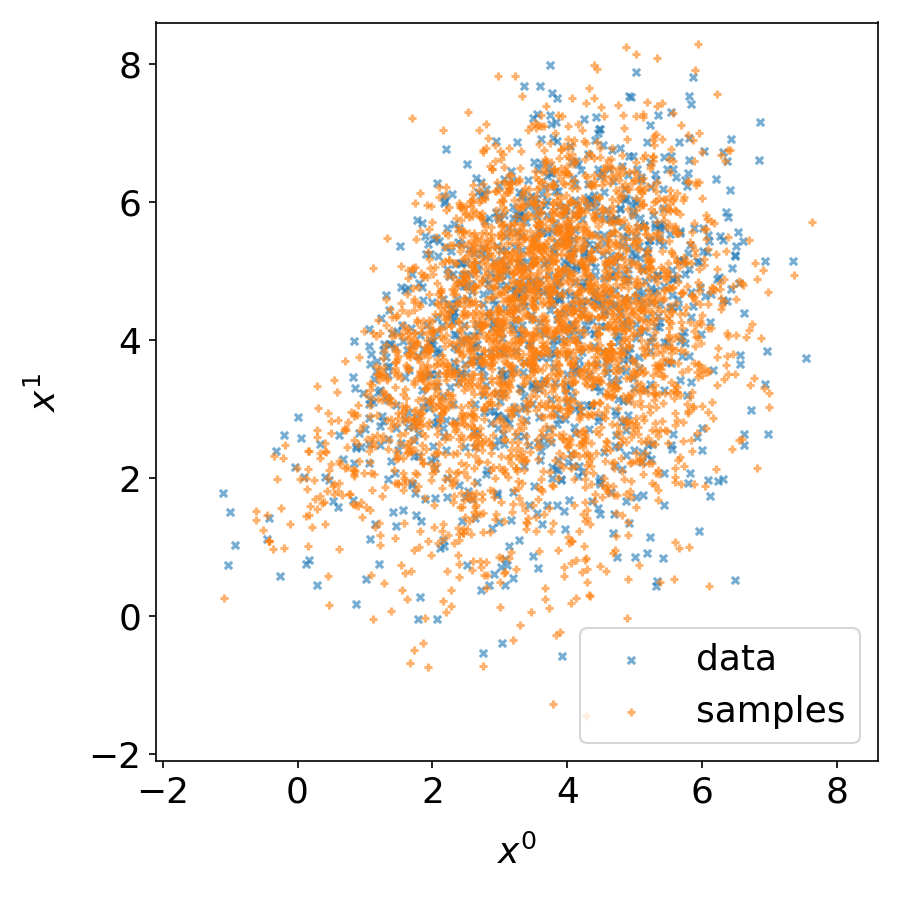}
    \end{minipage}
  }
  \caption{
    Scatter plot of $(x^0, x^1)$ for samples generated by the DK-SDE model and the given data.
    From left to right: samples generated by SDE with initial parameter, after 70 updates, after 500 updates.
  }
  \label{f:gm20d_scatter}
\end{figure}

Direct baselines are unavailable for this problem. Methods based on differentiating through trajectories (adjoint/tangent backpropagation) are numerically ineffective once the dynamics becomes unstable/nonhyperbolic (often already at moderate horizons in our setting), while likelihood-ratio constructions typically do not apply to parameters in the diffusion or suffer prohibitive variance. 
Other approaches (e.g.\ discrepancy or feature matching) optimize different objectives.



\section*{Declarations}

\oldsubsection*{Data availability statement}
The code used in this paper is posted at \url{https://github.com/niangxiu/DKlinR}.
There are no other associated data.

\oldsubsection*{Funding}
The author has no active funding.

\oldsubsection*{Conflicts of interests/Competing interests}
The author has no relevant conflicting or competing interests.


\begin{appendix}

\section{Kernel-differentiation method}
\label{s:kd}

We briefly go through the kernel-differentiation method, and explain its problem when applied to linear responses of high-dimensional systems with respect to diffusion coefficients.

\subsection{One-step results} \label{s:ker1step}

Let $x^\gamma_0\sim h^\gamma_0$, $b^\gamma_0\sim k$, the one-step dynamics is 
\begin{equation*}\begin{split}
  x^\gamma_1 = f^\gamma_0(x^\gamma_0) + \sigma^\gamma_0 (x^\gamma_0) b_0.
\end{split}\end{equation*}
The most natural expression of $h_1$ is:
\begin{equation*}
  h^\gamma_1(x_1) = \int h^\gamma_0(x_0) (\sigma^\gamma_0) ^{-M}(x_0) 
  k\left(\frac{x_1-f^\gamma_0(x_0)}{\sigma^\gamma_0(x_0)}\right) dx_0.
\end{equation*}
Take derivative with respect to $\gamma$, then the probability kernel $k$ will be differentiated, and we get the kernel-differentiation formula for one-step.

\begin{lemma} [Kernel formula for one-step linear response] \label{t:ker1step}
\begin{equation*}
  \delta \log h^\gamma_1(x_1) 
  = \E{\delta \log h^\gamma_0(x_0) 
  - M \delta \log \sigma^\gamma_0(x_0)
  - \frac {\nabla \log k(b_0)}{\sigma_0(x_0)}
  \left( \delta \sigma^\gamma_0(x_0)b_0
  + \delta f^\gamma(x_0)
  \right)
  \middle| x_1 }.
\end{equation*}
\end{lemma}

\begin{proof}
Take derivative of \Cref{e:h1} with respect to $\gamma$, 
\begin{equation*}\begin{split}
  \delta h^\gamma_1(x_1) 
  = \int \left(
  \frac{\delta h^\gamma_0}{h_0}(x_0)
  - M \frac {\delta \sigma^\gamma_0} {\sigma_0}(x_0)
  - \frac {\nabla k} {k} (b_0) 
  \left(
  \frac {\delta \sigma^\gamma_0} {\sigma_0}(x_0) b_0
  + \frac {\delta f^\gamma_0} {\sigma_0}(x_0)
  \right)
  \right)
  h_0(x_0) \frac{k\left(b_0\right)} {\sigma^{M}(x_0)} dx_0,
\end{split}\end{equation*}
where $b_0=(x_1-f_0(x_0)) / \sigma_0(x_0)$.
Divide both sides by $h_1(x_1)$ and apply the definition of the conditional expectation to prove the lemma.
\end{proof}

The second term in \Cref{t:ker1step} can be very large for high-dimensional noise since it involves $M$.
For now, we still use noise in all directions, so we try to avoid using the kernel formula for parameter-derivative.
The kernel formula for the spatial derivative does not have terms involving $M$ \cite{divKer}, so we can still use that.
We will try to use low-dimensional noise in later papers, which will incur additional difficulties.

\subsection{Discrete-time many-steps }

First by recursively apply \Cref{t:ker1step}, we can prove the following kernel-differentiation formula for linear response.
Here, if we set $\delta \sigma =0$, this becomes the typical likelihood ratio formula.

\begin{lemma} [N-step kernel-differentiation formula for linear responses] \label{t:kernstep}
\begin{equation*}
  \delta \log h^\gamma_N(x_N) 
  = \E{\delta \log h^\gamma_0(x_0) 
  - \sum_{n=0}^{N-1}
  \frac {1}{\sigma_n(x_n)}
  \left(
  M \delta \sigma^\gamma_n(x_n)
  + \nabla \log k(b_n)
  \left( \delta \sigma^\gamma_n(x_n)b_n
  + \delta f^\gamma(x_n)
  \right)
  \right)
  \middle| x_N }.
\end{equation*}
\end{lemma}

\subsection{Formal continuous-time limit}

For the continuous time limit, we first consider the time-discretized SDEs.
Note that if $\delta \sigma\neq 0$, then the two terms involve $\delta \sigma$ in the kernel formula in \Cref{t:ker1step} becomes $O(1)$ in each step, so the integration in time does not converge --
this is why likelihood ratio method always uses additive noise independent of the parameter $\gamma$.
For the degenerate case where $\delta \sigma\equiv0$, \Cref{t:kernstep} converges to the classical likelihood ratio method in continuous time.

\begin{lemma} [likelihood ratio formula for linear response of SDEs] \label{t:ker_sde}
If $\delta \sigma\equiv0$,
\begin{equation*}\begin{split}
  \delta \log h_T(x_T)
  =
  \E{\delta \log h^\gamma_0(x_0) 
  + \int_{t=0}^{T}
  \frac {1}{\sigma_t(x_t)}
  \delta F^\gamma(x_t)
  \cdot dB
  \middle| x_T }.
\end{split}\end{equation*}
\end{lemma}

\end{appendix}

\bibliographystyle{abbrv}
{\footnotesize\bibliography{library}}

\end{document}